\documentclass[10pt]{article}
\usepackage{graphicx}
\usepackage{indentfirst,csquotes}
\usepackage{graphicx}
\usepackage[colorlinks, linkcolor=blue, anchorcolor=blue, citecolor=blue]{hyperref}
\usepackage[title]{appendix}
\usepackage{mathrsfs}
\usepackage{amsfonts}
\usepackage{booktabs} 
\usepackage{caption}  
\usepackage{threeparttable} 
\usepackage{algorithm}   
\usepackage{algorithmicx}
\usepackage{algpseudocode}  
\usepackage{listings}
\usepackage{enumitem}
\usepackage{booktabs}
\usepackage{lipsum}
\usepackage{subcaption}
\usepackage[T1]{fontenc}    
\usepackage{csquotes}       
\usepackage{diagbox}
\usepackage{comment}
\usepackage{amsmath, amssymb,amsthm,color,enumerate,comment,graphicx,float,appendix,amssymb,enumitem,arydshln,mathrsfs}
\newcommand{\no}[1]{#1}

\renewcommand{\no}[1]{}
\no{\usepackage{times}\usepackage[subscriptcorrection, slantedGreek, nofontinfo]{mtpro}
	\renewcommand{\Delta}{\upDelta}}

\title{Curvilinear Mask Optimization for Inverse Lithography Based on B-splines and Delaunay Triangulation\thanks{The work was supported by National Key  R\&D Program of China 2019YFA0709600, 2019YFA0709602. }}
\date{}
\author{Xiaoru Yi\thanks{Department of Mathematical Sciences, Tsinghua University, Beijing 100084, China. (\texttt{yixr23@mails.tsinghua.edu.cn, jqchen@tsinghua.edu.cn})}\and Junqing Chen\footnotemark[2]}

\date{\today}
\newtheorem{theorem}{Theorem}[section]

\newtheorem{lemma}{Lemma}[section]
\newtheorem{definition}{Definition}[section]

\newtheorem{remark}{Remark}[section]

\numberwithin{equation}{section}

\renewcommand{\geq}{\geqslant}

\begin{document}
	
	\maketitle
	\begin{abstract}
		In this paper, we propose a novel gradient-based method to optimize curvilinear masks in optical lithography. The mask pattern is represented by periodic B-spline curves. 
		We apply Delaunay triangulation to discretize the domains circled by the spline curves. Subsequently, we establish an explicit relationship between the integral points and the control points of the boundary spline curve. Based on the relationship, we derive explicit formulas of the gradient of the optimization objective function with respect to the coordinates of the control points. Then we propose an inverse lithography algorithm to optimize the curvilinear mask pattern. Finally, the results of the numerical experiments demonstrate the feasibility and extensive adaptability of our method.
		
		{\bf Key words}:curvilinear masks, periodic B-spline curve, inverse lithography techniques, gradient-based optimization, optical proximity correction
	\end{abstract}
	
	\maketitle

\section{Introduction}\label{introduction}
As a crucial part of semiconductor manufacturing, projection lithography technology is tasked with transferring the patterns on the photomask onto the wafer coated with photoresist. During this process, patterns with critical size several times the lithographic wavelength generally do not require any correction. However, as the critical size decreases to be comparable or even smaller than the wavelength of the lithography light source, owing to the diffraction and interference effects of adjacent patterns, the patterns on the wafer will experience varying degrees of distortion, namely the so-called optical proximity effect (OPE). 
	
In order to overcome the distortion from the perspective of the mask pattern, the optical proximity correction (OPC) technology has been proposed. By making certain corrections to the mask pattern, the pattern obtained on the wafer can be made to more closely resemble the designed pattern. In the early days, most OPC technologies are rule-based OPC. Methods of this kind set the pattern correction rules based on the operational experience of lithography engineers to correct simple mask patterns. But as the size of the mask pattern further decreases, the proximity effect becomes more severe and the correction rules for the mask pattern increase exponentially, even becoming unfeasible. To overcome this difficulty, Rieger and Stirniman \cite{Rieger1}, and Nick Cobb \cite{Cobb1} first proposed the model-based OPC method. The earliest OPC method was based on the Manhattan mask, which only makes fine-tuning of the segmented boundaries along specific angles, so the flexibility of this approach is relatively limited. 
	
To extend the flexibility of OPC and obtain a better mask pattern, pixel-based inverse lithography technology (ILT) was proposed. This method is highly flexible and enables more precise adjustments to the photomask, thus resulting in a superior mask quality. Saleh and Sayegh explored optimized photomasks by "pixel flipping" in 1981 \cite{1981}. Liu and Zakhor used branch and bound and the simplex method \cite{liu1} and the "bacteria" algorithm \cite{liu2} to solve the problem. Granik considered an objective function with linear, quadratic, and nonlinear formulations \cite{granik}. Ma and his collaborators extended the pixel-based optimization method from the scalar to the vector imaging model \cite{maxu1,maxu2} and introduced the source and mask optimization (SMO) in \cite{maxu3}. However, this method has some inherent drawbacks. For example, it introduces a large number of optimization variables, and the optimized mask pattern is relatively fragmented and difficult to manufacture in practice. Some regularization methods have been proposed to constrain the mask pattern. Choy et al. developed a robust computational algorithm for the loss function with multiple regularization terms \cite{robust}. Liu and Chen applied the ADMM method to the inverse lithography and performed convergence analysis \cite{liuhaibo}. A fuller review of inverse lithography techniques could be found in \cite{Inverse30,advances}.  
	
Variable shaped beam (VSB) mask writers are widely used in mask manufacturing. Since they cannot handle complex mask patterns, the vast majority of mask designs have to be simplified to the Manhattan pattern for production. This is also the reason why curvilinear masks were not highly regarded in the early days. With the introduction of multi-beam mask writers, it has become possible to produce curvilinear masks. Pang and Fujimura reviewed the trend and benefits of OPC methods based on curvilinear mask in \cite{CMwhy}. In recent years, many OPC methods based on curvilinear masks have emerged. For example, Chen et al. extended traditional edge-based OPC to vertices and produced curvilinear OPC results by following the conventional OPC workflow in \cite{Yung-YuChen}. Huang et al. proposed distance-versus-angle signature (DVAS), a one-dimensional function, to represent a two-dimensional boundary of the mask and optimized the shape in \cite{huangweichen}. Yang et al. proposed a curvilinear OPC method adopting the quasi-uniform B-spline curve and non-uniform B-spline curve to improve the optimization efficiency in \cite{heyang1,heyang2}. Wei et al. established an optimization method based on an implicit function to reduce the number of variables in \cite{wei}. Other attempts on the optimization of curvilinear mask can be found in \cite{other1,other2,other3,other4,other5}. Machine learning models are also applied for accurate curve correction \cite{machine1,machine2,machine3}. For further information about the curvilinear mask, we recommend reading \cite{CMmotivations,CMoverview}.
	
In this paper, we propose a novel inverse lithography method based on periodic B-splines to construct the mask boundary. Compared with traditional pixel-based inverse lithography techniques, this approach can greatly reduce the number of optimization variables. Moreover, the spline curve itself is one of the standard formats in computer-aided design(CAD) \cite{CAD1}, and the optimized mask shape has good continuity and high manufacturability. We use the spline boundary to generate a triangular mesh by Delaunay triangulation and creatively establish an explicit relationship between the Gaussian integral points and the control points. Based on the photoresist imaging model, we establish an explicit gradient formula for optimization of control points, which can be generalized to a variety of gradient-based optimization algorithms. We also provide some numerical examples to verify the feasibility of our method.
	
The structure of this paper is as follows. 
In Section \ref{preliminary}, the basic knowledge about projection lithography and periodic B-splines required in this paper is introduced. 
In Section \ref{formula}, the numerical calculation formula for the forward lithography problem is presented, an explicit relationship between the integration points and the control points is established, and some important error bounds are estimated. Subsequently, based on the photoresist imaging model, an inverse lithography problem model is given and an explicit gradient formula is derived. 
In Section \ref{algorithm}, a basic description of the optimization algorithm is provided. 
In Section \ref{numerical}, some numerical examples are given to verify the reliability of our method.

\section{Preliminary}\label{preliminary}
\subsection{Basic theory of optical lithography}
Projection microlithography is the key to the semiconductor production. Normally, the imaging system implements K\"ohler's illumination, details of which can be found in \cite{principleofoptics}. A basic representation of an optical lithography imaging system is shown in Figure \ref{Figkohler}. When light is emitted from the source, it is condensed by the condenser lens $L_c$ and becomes parallel light. The parallel light illuminates the mask uniformly and carries the information of mask pattern. Then the light is modulated by the projection lens, finally converges on the image plane, and prints the mask pattern onto the wafer. In the procedure, the light will be limited by the pupil size and lose some pattern information of high frequency. The light will be scattered by the mask pattern as well, then the image on the wafer is generally different from that on the mask, so we need to take some measures to correct the pattern on the mask.
	\begin{figure}[htbp]
		\centering
		\includegraphics[width=0.7\textwidth]{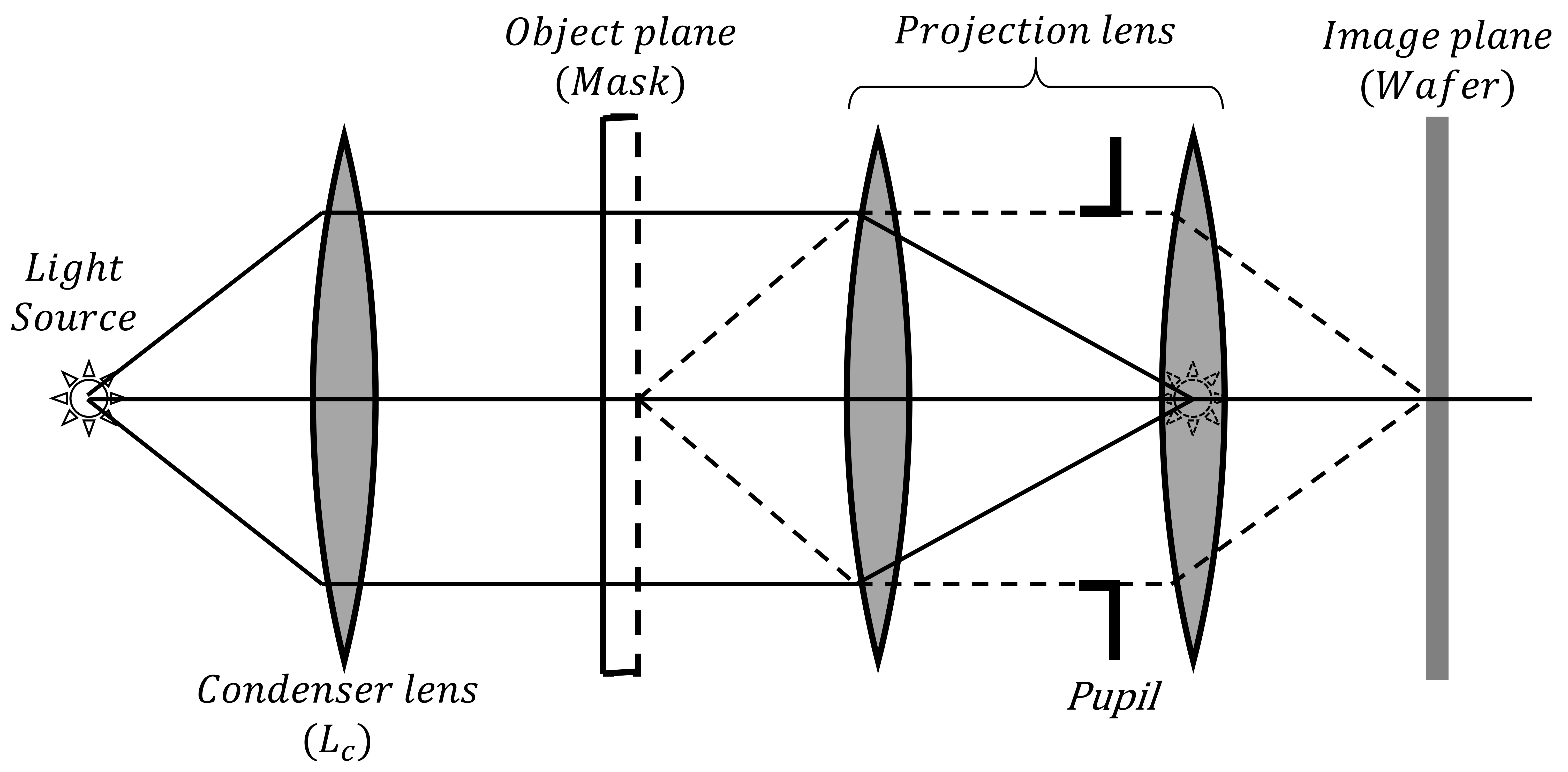}
		\caption{Basic representation of an optical lithography imaging system} 
		\label{Figkohler} 
	\end{figure}

	According to the work of \cite{OpticalImaging}, the image intensity for an isoplanatic region under K\"ohler illumination is
	\begin{equation*}
		\hat{I}(\hat{x}_i,\hat{y}_i)=\iiiint\limits_{-\infty}^{+\infty}TCC(\hat{f}^{\prime},\hat{g}^{\prime};\hat{f}^{\prime\prime},\hat{g}^{\prime\prime})\tilde{\hat{O}}(\hat{f}^{\prime},\hat{g}^{\prime})\tilde{\hat{O}}^{*}(\hat{f}^{\prime\prime},\hat{g}^{\prime\prime})e^{-i2\pi[(\hat{f}^{\prime}-\hat{f}^{\prime\prime})\hat{x}_i+(\hat{g}^{\prime}-\hat{g}^{\prime\prime})\hat{y}_i]}d\hat{f}^{\prime}d\hat{g}^{\prime}d\hat{f}^{\prime\prime}d\hat{g}^{\prime\prime},
	\end{equation*}
	where $TCC(\hat{f}^{\prime},\hat{g}^{\prime};\hat{f}^{\prime\prime},\hat{g}^{\prime\prime})$ is the transmission cross-coefficient
	\begin{equation*}	TCC(\hat{f}^{\prime},\hat{g}^{\prime};\hat{f}^{\prime\prime},\hat{g}^{\prime\prime})=\iint\limits_{-\infty}^{+\infty}\tilde{\hat{J}}(\hat{f},\hat{g})\tilde{\hat{H}}(\hat{f}+\hat{f}^{\prime},\hat{g}+\hat{g}^{\prime})\tilde{\hat{H}}^{*}(\hat{f}+\hat{f}^{\prime\prime},\hat{g}+\hat{g}^{\prime\prime}) d\hat{f}d\hat{g},
	\end{equation*}
and	$\tilde{\hat{O}},\tilde{\hat{J}},\tilde{\hat{H}}$ represent the Fourier integral of $\hat{O},\hat{J},\hat{H}$:
	\begin{equation*}
		\hat{O}(\hat{x}_o,\hat{y}_o)=\iint\limits_{-\infty}^{+\infty}\tilde{\hat{O}}(\hat{f},\hat{g})e^{-i2\pi(\hat{f}\hat{x}_o+\hat{g}\hat{y}_o)}d\hat{f}d\hat{g},
	\end{equation*}
	\begin{equation*}
		\hat{J}(\hat{x}_o'-\hat{x}_o'',\hat{y}_o'-\hat{y}_o'')=\iint\limits_{-\infty}^{+\infty}\tilde{\hat{J}}(\hat{f},\hat{g})e^{-i2\pi(\hat{f}(\hat{x}_o'-\hat{x}_o'')+\hat{g}(\hat{y}_o'-\hat{y}_o''))}d\hat{f}d\hat{g},
	\end{equation*}
	\begin{equation*}
		\hat{H}(\hat{x}_i-\hat{x}_o^{\prime},\hat{y}_i-\hat{y}_o^{\prime})=\iint\limits_{-\infty}^{\infty}\tilde{\hat{H}}(\hat{f}+\hat{f}^{\prime},\hat{g}+\hat{g}^{\prime})e^{-i2\pi((\hat{f}+\hat{f}')(\hat{x}_i-\hat{x}_o')+(\hat{g}+\hat{g}')(\hat{y}_i-\hat{y}_o'))}d\hat{x}_o^{\prime}d\hat{y}_o^{\prime}.
	\end{equation*}
The normalized variables that appear above are listed below
	\begin{equation}\label{normalize}
		\begin{aligned}
			\hat{x}_{o}&=-\frac{Mx_{o}}{\lambda_{0}/\mathrm{NA}}, &\hat{y}_{o}&=-\frac{My_{o}}{\lambda_{0}/\mathrm{NA}},\\
			\hat{x}_{i}&=\frac{x_{i}}{\lambda_{0}/\mathrm{NA}}, &\hat{y}_{i}&=\frac{y_{i}}{\lambda_{0}/\mathrm{NA}},\\
			\hat{f}&=\frac{f}{\mathrm{NA}/\lambda_{0}}, &\hat{g}&=\frac{g}{\mathrm{NA}/\lambda_{0}}. 
		\end{aligned}
	\end{equation}
	The corresponding inverse normalization are
	\begin{equation}\label{inverse normalization}
		\begin{aligned}
			x_o&=-\frac{\lambda_0\hat{x}_o}{M\mathrm{NA}}, &y_o&=-\frac{\lambda_0\hat{y}_o}{M\mathrm{NA}},\\
			x_i&=\frac{\lambda_0\hat{x}_i}{\mathrm{NA}}, &y_i&=\frac{\lambda_0\hat{y}_i}{\mathrm{NA}},\\
			f&=\frac{\hat{f}\mathrm{NA}}{\lambda_0}, &g&=\frac{\hat{g}\mathrm{NA}}{\lambda_0}.
		\end{aligned}
	\end{equation}
	Here $(x_o,y_o)$ represent the original spatial coordinates on the object plain, $(\hat{x}_{o},\hat{y}_{o})$ represent the normalized spatial coordinates on the object plain. $(x_i,y_i)$ and $(\hat{x}_{i},\hat{y}_{i})$ represent the original spatial coordinates and its normalization on the image plain. $f,g$ represent spatial frequencies and $\hat{f},\hat{g}$ are normalized spatial frequencies.
	$M$ is the lateral magnification, $\lambda_0$ is the wave length and NA is the numerical aperture of the imaging system. $\hat{I}(\hat{x}_i,\hat{y}_i)$ is the intensity at a point $(\hat{x}_i,\hat{y}_i)$ on the image plane. $\hat{J}(\hat{x}_o^{\prime}-\hat{x}_o^{\prime\prime},\hat{y}_o^{\prime}-\hat{y}_o^{\prime\prime})$ is the mutual intensity between a pair of illuminated points $(\hat{x}_o^{\prime},\hat{y}_o^{\prime})$ and $(\hat{x}_o^{\prime\prime},\hat{y}_o^{\prime\prime})$. $\hat{O}(\hat{x}_o^{\prime},\hat{y}_o^{\prime})$ is the function used to describe the mask pattern. Under the assumption that the imaging region is so small that it forms an isoplanatic region, $\hat{H}(\hat{x}_i-\hat{x}_o,\hat{y}_i-\hat{y}_o)$ is the normalized version of the optical system transfer function.
	
	In this paper, we assume that the illumination is coherent and the effective source is a point at $(\hat{f}_0,\hat{g}_0)$, namely $\tilde{\hat{J}}(\hat{f},\hat{g})=\delta(\hat{f}-\hat{f}_0,\hat{g}-\hat{g}_0)$, so that we can derive the coherent image intensity
	\begin{equation}\label{intensity}
		\hat{I}(\hat{x}_i,\hat{y}_i)=|\hat{U}(\hat{x}_i,\hat{y}_i)|^2,
	\end{equation}
with
    \begin{equation*}
	\hat{U}(\hat{x}_i,\hat{y}_i)=\iint\limits_{-\infty}^{+\infty}\tilde{\hat{H}}(\hat{f},\hat{g})\tilde{\hat{O}}(\hat{f}-\hat{f}_0,\hat{g}-\hat{g}_0)e^{-i2\pi(\hat{f}\hat{x}_i+\hat{g}\hat{y}_i)}d\hat{f}d\hat{g},
	\end{equation*}
	where $\hat{U}(\hat{x}_i,\hat{y}_i)$ represents the amplitude of a point on the image plane under a normalized variable representation. The formula above implies:
	\begin{equation*}
		\tilde{\hat{U}}(\hat{f},\hat{g})=\tilde{\hat{O}}(\hat{f}-\hat{f}_0,\hat{g}-\hat{g}_0)\tilde{\hat{H}}(\hat{f},\hat{g}).
	\end{equation*}
	According to the convolution theorem, if we assume that $(\hat{f}_0,\hat{g}_0)=(0,0)$, we can get the following formula
	\begin{equation}\label{imagingformula}
		\hat{U}(\hat{x}_i,\hat{y}_i)=\iint\limits_{-\infty}^{+\infty}\hat{O}(\hat{x}_o,\hat{y}_o)\hat{H}(\hat{x}_i-\hat{x}_o,\hat{y}_i-\hat{y}_o) d\hat{x}_od\hat{y}_o.
	\end{equation}
From \cite[equation (4.10)]{OpticalImaging} , it is known that 
	\begin{equation*}
		\hat{H}(\hat{x}_i-\hat{x}_o,\hat{y}_i-\hat{y}_o)=\iint\limits_{-\infty}^{+\infty}\hat{\tilde{P}}(\hat{f},\hat{g})e^{-i2\pi(\hat{f}(\hat{x}_i-\hat{x}_o)+\hat{g}(\hat{y}_i-\hat{y}_o))}d\hat{f}d\hat{g},
	\end{equation*}
	where $\hat{\tilde{P}}(\hat{f},\hat{g})$ is the pupil function under the normalized coordinates, which is:
	\begin{equation*}
		\hat{\tilde{P}}(\hat{f},\hat{g})=
		\begin{cases}
			1&\quad\text{if}\sqrt{\hat{f}^2+\hat{g}^2}\leqslant1,\\
			0&\quad\text{otherwise.}
		\end{cases}
	\end{equation*}
    Here we define $circ$ function in the following way:
    \begin{equation*}
        circ(\sqrt{x^2+y^2})=
        \begin{cases}
        1,\quad &if ~\sqrt{x^2+y^2}<1,\\
        1/2, \quad &if ~\sqrt{x^2+y^2}=1,\\
        0,\quad &else.
        \end{cases}
    \end{equation*}
	Since the set in which the value of the circ function is 1/2 is a set of measure of zero, we can replace $\hat{\tilde{P}}(\hat{f},\hat{g})$ with $circ(\sqrt{\hat{f}^2+\hat{g}^2})$ without making any difference to the value of $\hat{H}(\hat{x}_i-\hat{x}_o,\hat{y}_i-\hat{y}_o)$,
	so we can get
	\begin{equation*}
		\hat{H}(\hat{x}_i-\hat{x}_o,\hat{y}_i-\hat{y}_o)=\iint\limits_{-\infty}^{+\infty}circ(\sqrt{\hat{f}^2+\hat{g}^2})e^{-i2\pi(\hat{f}(\hat{x}_i-\hat{x}_o)+\hat{g}(\hat{y}_i-\hat{y}_o))}d\hat{f}d\hat{g}.
	\end{equation*}
This means that $\hat{H}(\hat{x}_i-\hat{x}_o,\hat{y}_i-\hat{y}_o)$ is in fact the standard Fourier transform of $circ$ function. Here we give the result of transformation directly:
	\begin{equation}\label{H}
		\hat{H}(\hat{x}_i-\hat{x}_o,\hat{y}_i-\hat{y}_o)=\frac{J_1(2\pi\rho)}{\rho}=\frac{J_1(2\pi\sqrt{(\hat{x}_i-\hat{x}_o)^2+(\hat{y}_i-\hat{y}_o)^2})}{\sqrt{(\hat{x}_i-\hat{x}_o)^2+(\hat{y}_i-\hat{y}_o)^2})},
	\end{equation}
	where $J_1$ is a first-order Bessel function of the first kind. For detailed derivation process, we recommend reading \cite[Chapter 2.1]{FourierOptics}
	
	\subsection{Periodic B-spline curve boundary representation}\label{spline boundary}
	In practical lithography for high-volume manufacturing, people always hope to find a relatively regular pattern which is easy to produce. Benefiting from the renewal of production equipment, it is possible to manufacture curvilinear masks whose pattern boundary consists of flexible curve instead of only axis-parallel rectangles or 45-degree triangles\cite{CMoverview}. A natural assumption is that the boundary of the curvilinear mask pattern is continuous. We use $\Omega$ to represent the transparent part of the mask and the boundary of $\Omega$ can be written in $\partial \Omega$. Moreover, since the transparent part of the mask is closed, the boundary curve $\partial\Omega$ should be periodic. If we limit the parameter of the boundary curve into interval $[0,1]$, and consider that the transparent part of the mask is composed of several close domains disconnected from each other, the boundary of curvilinear mask can be described as 
	\begin{align*}
		&\partial \Omega=\{\textbf{g}_1(t),\textbf{g}_2(t),\cdots,\textbf{g}_n(t)\},\\        &\textbf{g}_k(t)=x_k(t)\mathbf{i}+y_k(t)\mathbf{j},~k=1,2,\cdots,n,\\
        &x_k(t),y_k(t)\in C([0,1]), ~x_k(0)=x_k(1), ~y_k(0)=y_k(1).
	\end{align*}
Here $\mathbf{i,j}$ are the unit vectors in $x$ and $y$ directions respectively.

    In the followings, we will use the B-spline function to represent the curve boundary of $\Omega$ because of the powerful expressive ability of B-spline. Since the boundary is periodic, it is natural to use the periodic B-spline function.
    First, we will give some basic notions about the B-spline curve. For a better understanding of the theory of the spline function, we recommend reading \cite{spline,splinefunctions,Buffaisogeometric}.


	B-spline curve is a kind of widely used polynomial spline curve, the basis of which is defined recursively:
\begin{definition}\label{B-spline}
	Given a knot vector on the interval $[a,b]$ written $\varXi=\{\xi_1,\xi_2,\cdots,\xi_{n+p+1}\}$, where $\xi_1=a, \xi_{n+p+1}=b, \xi_{k}\leqslant\xi_{k+1}$, and $\xi_k\in\mathbb{R}$ is the $k$-th knot, $k=1,2,\cdots,n+p+1$, $p$ is the polynomial degree, and $n$ is the number of basis functions that comprise the B-spline. Then k-th B-spline basis function of degree p $N_{k,p}$ can be defined recursively according to the following de Boor-Cox formula for $p=0,1,2,...$ :
	\begin{equation*}
		\left\{\begin{aligned}
			&N_{k,0}(\xi)=\left\{\begin{aligned}1, \quad &\xi\in[\xi_k,\xi_{k+1}),\\0, \quad &\text{else,}\end{aligned}\right.\\
			&N_{k,p}(\xi)=\frac{\xi-\xi_k}{\xi_{k+p}-\xi_k}N_{k,p-1}(\xi)+\frac{\xi_{k+p+1}-\xi}{\xi_{k+p+1}-\xi_{k+1}}N_{k+1,p-1}(\xi),\quad p\geqslant1,\\
			&\text{define}\quad\frac{0}{0}=0.
			\end{aligned}\right.
	\end{equation*}
\end{definition}	
With the definition of B-spline basis, we can further define the B-spline curves:
	\begin{definition}
		
		Given $n$ space vectors $\{\mathbf{P}_k\}_{k=1}^{n}\in\mathbb{R}^d$, and assuming that $N_{k,p}(\xi)$ are the B-spline basis functions of degree $p$ defined on the knot vector $\varXi=\{\xi_1,\xi_2,\cdots,\xi_{n+p+1}\}(\xi_k\leqslant\xi_{k+1},k=1,2,\cdots,n+p)$, we can define:
		\begin{equation*}
			\mathbf{P}(\xi)=\sum_{k=1}^{n}\mathbf{P}_{k}N_{k,p}(\xi),\quad \xi\in[\xi_{p},\xi_{n+1}],
		\end{equation*} 
		is the $p$-th order B-spline curve correspond to knot vector $\varXi$.
	\end{definition}

To generate a periodic B-spline curve, we first introduce the extended partition.
\begin{definition}\label{extpart}
	For a knot vector on the interval $[a,b]$ written $\varXi=\{\xi_1,\xi_2,\cdots,\xi_{n+p+1}\}$, we can define its extended partition $\bar{\varXi}$ as follows:
	\begin{equation*}
		\bar{\varXi}=\{\xi_{1-p},\cdots,\xi_1,\xi_2,\cdots,\xi_{n+p+1},\cdots,\xi_{n+2p+1}\}.
	\end{equation*}
	Here we define:
	\begin{equation*}
		\left\{
		\begin{aligned}
			&\xi_{k}=\xi_{k+n+p}-L,&1-p\leqslant&k<1,\\
			&\xi_{k}=\xi_{k-n-p}+L,&n+p+1< &k\leqslant n+2p+1,
		\end{aligned}
		\right.
	\end{equation*}
	where $L=b-a$.
\end{definition}
	
	Based on the extended partition in Lemma \ref{extpart}, we can further define the periodic B-spline basis function by the following definition.
	\begin{definition}
		For a given extended partition $\bar{\varXi}=\{\xi_{1-p},\cdots,\xi_1,\xi_2,\cdots,\xi_{n+p+1},\cdots,\xi_{n+2p+1}\}$ of the interval $[a,b]$, 
        we can define the $k$-th periodic B-spline basis function of degree $p$ by
		\begin{equation*}
			\mathring{N}_{k,p}(\xi)=\left\{
			\begin{aligned}
				&N_{k,p}(\xi) \quad k=1,2,\cdots,n,\\
				&N_{k,p}(\xi)+N_{k-n-p,p}(\xi) \quad k=n+1,n+2,\cdots,n+p,
			\end{aligned}
			\right.
		\end{equation*}
	\end{definition}
	Here, the definition is a little different from that in \cite{splinefunctions}, but in fact they are the same.
	We can similarly define the periodic B-spline curve as follows.
	\begin{definition}
			Given $n+p$ space vectors $\{\mathbf{P}_k\}_{k=1}^{n+p}\in\mathbb{R}^d$, and assuming that $\mathring{N}_{k,p}(\xi)$ are the periodic B-spline basis functions of degree $p$  defined on the extended partition
            \begin{equation*}
            \bar{\varXi}=\{\xi_{1-p},\cdots,\xi_1,\xi_2,\cdots,\xi_{n+p+1},\cdots,\xi_{n+2p+1}\}.
            \end{equation*}
           Then the periodic B-spline curve of degree $p$ correspond to knot vector in  $\bar{\varXi}$ is defined by
		\begin{equation*}
			\mathring{\mathbf{P}}(\xi)=\sum_{k=1}^{n+p}\mathbf{P}_k\mathring{N}_{k,p}(\xi),\quad \xi\in[\xi_1,\xi_{n+p+1}].
		\end{equation*} 
	 And we call $\mathbf{P}_k, ~k=1,2,...,n+p$ the control points of periodic B-spline curve.
	\end{definition}
	
	With the definition of periodic B-spline curve, we can use it to describe the boundary of mask pattern. Theorem 8.12 in \cite{splinefunctions} guarantees that for any periodic continuous boundary curve, it can be approximated well by periodic B-spline curves. With the increasing of number of control points and the decreasing of $\bar{\Delta}=\max\limits_{1\leqslant k \leqslant n+p}\{\xi_{k+1}-\xi_{k}\}$, the distance between a continuous periodic function and its periodic B-spline approximation can be arbitrarily small, which means that it is reasonable to represent the boundary of mask pattern by periodic B-spline curve.
	
	\section{Forward and inverse lithography based on B-spline function}\label{formula}
	\subsection{Region approximation and discretization}
	In this part, for the sake of simplicity, we omit the symbol~ $\hat{ }$ ~which represents the normalization of function and variable, and rewrite the imaging formulas \eqref{intensity} and \eqref{imagingformula} as
	\begin{equation}\label{imaging1}
		U(x_i,y_i)=\iint\limits_{-\infty}^{+\infty}O(x_o,y_o)H(x_i-x_o,y_i-y_o) dx_ody_o,\quad I(x_i,y_i)=|U(x_i,y_i)|^2,
	\end{equation}
	where $O(x_o,y_o)$ describes the pattern on the mask. In the following, we represent the pattern by binary function
	\begin{equation*}
		O(x_o,y_o)=
		\begin{cases}
			1&\quad\text{if location} (x_o,y_o) \text{ on the mask is transparent},\\
			0&\quad\text{if location} (x_o,y_o) \text{ on the mask is not transparent}.
		\end{cases}
	\end{equation*}
	We denote the transparent part of the curvilinear mask by $S_0$, whose boundary is represented by B-spline functions introduced in Section \ref{spline boundary}. Under the binary assumption about the mask pattern, the imaging formula \eqref{imaging1} can be rewritten as:
	 \begin{equation}\label{imaging2}
	 	U(x_i,y_i)=\iint\limits_{S_0}H(x_i-x_o,y_i-y_o) dx_ody_o,\quad I(x_i,y_i)=|U(x_i,y_i)|^2.
	 \end{equation}
     
	To calculate the integral in domain $S_0$, we have to approximate the domain by polygons and discretize the polygons by triangulation. The procedure is shown in Figure \ref{polygon approximation and triangulation}, so that the formula \eqref{imaging2} becomes:
	 \begin{equation}\label{U_triangle}
U(x_i,y_i)=\sum\limits_{p=1}^{N_T}\iint\limits_{S_{p}}H(x_i-x_o,y_i-y_o) dx_ody_o,\quad I(x_i,y_i)=|U(x_i,y_i)|^2,
	 \end{equation}
	 where $p$ is the index of small triangle, $N_T$ is the total number of triangles, $S_p$ is the $p$-th triangle whose measure is $|S_p|$. 
      $\sum\limits_{p=1}^{N_T}|S_p|=|S|$, which is the measure of the whole polygon.
\begin{figure}[htbp]
		\centering            
		\subfloat[Original domain]   
		{
			\label{fig21}\includegraphics[width=0.3\textwidth]{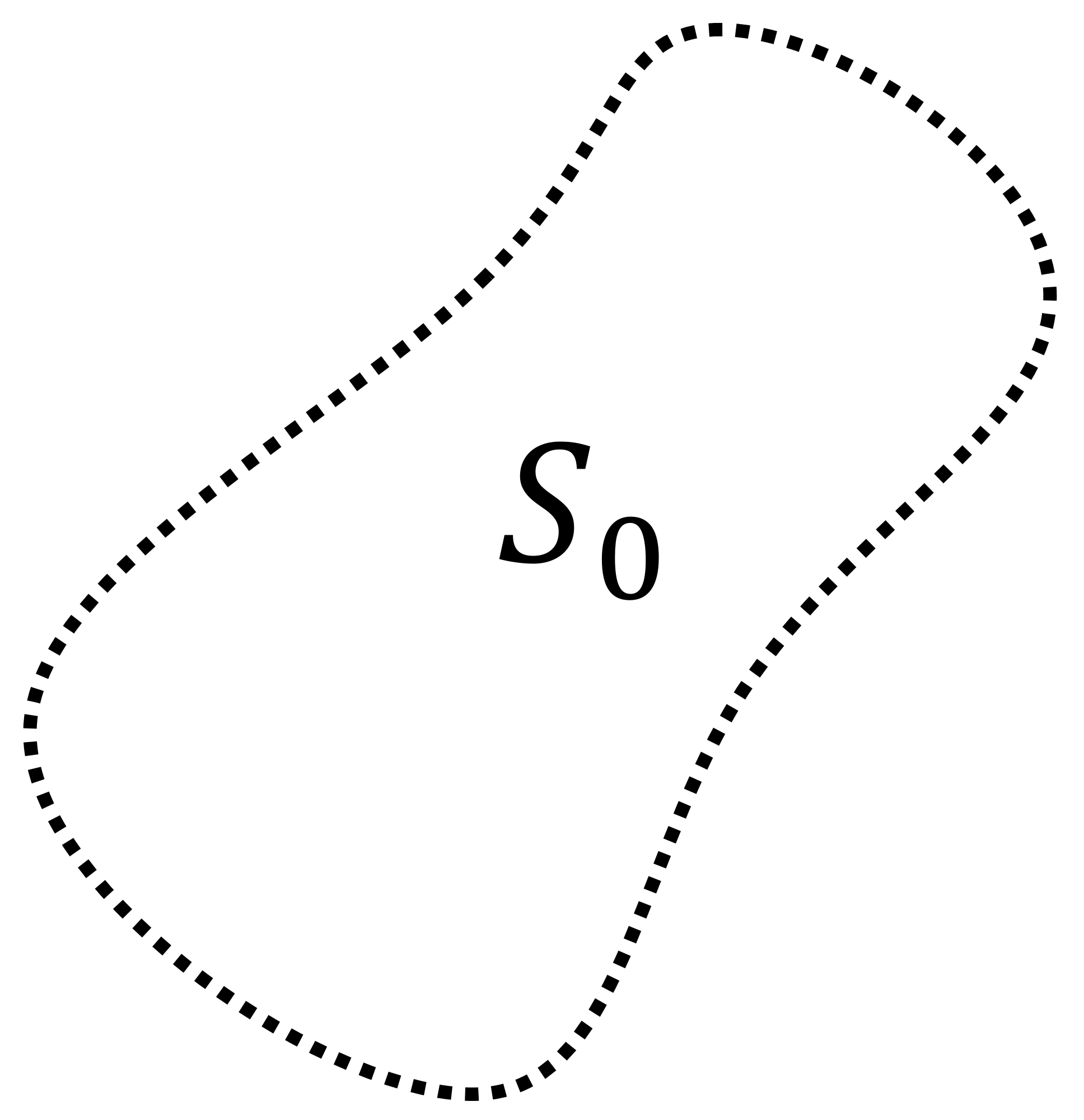}
		}
		\subfloat[Polygon approximation]   
		{
			\label{fig22}\includegraphics[width=0.3\textwidth]{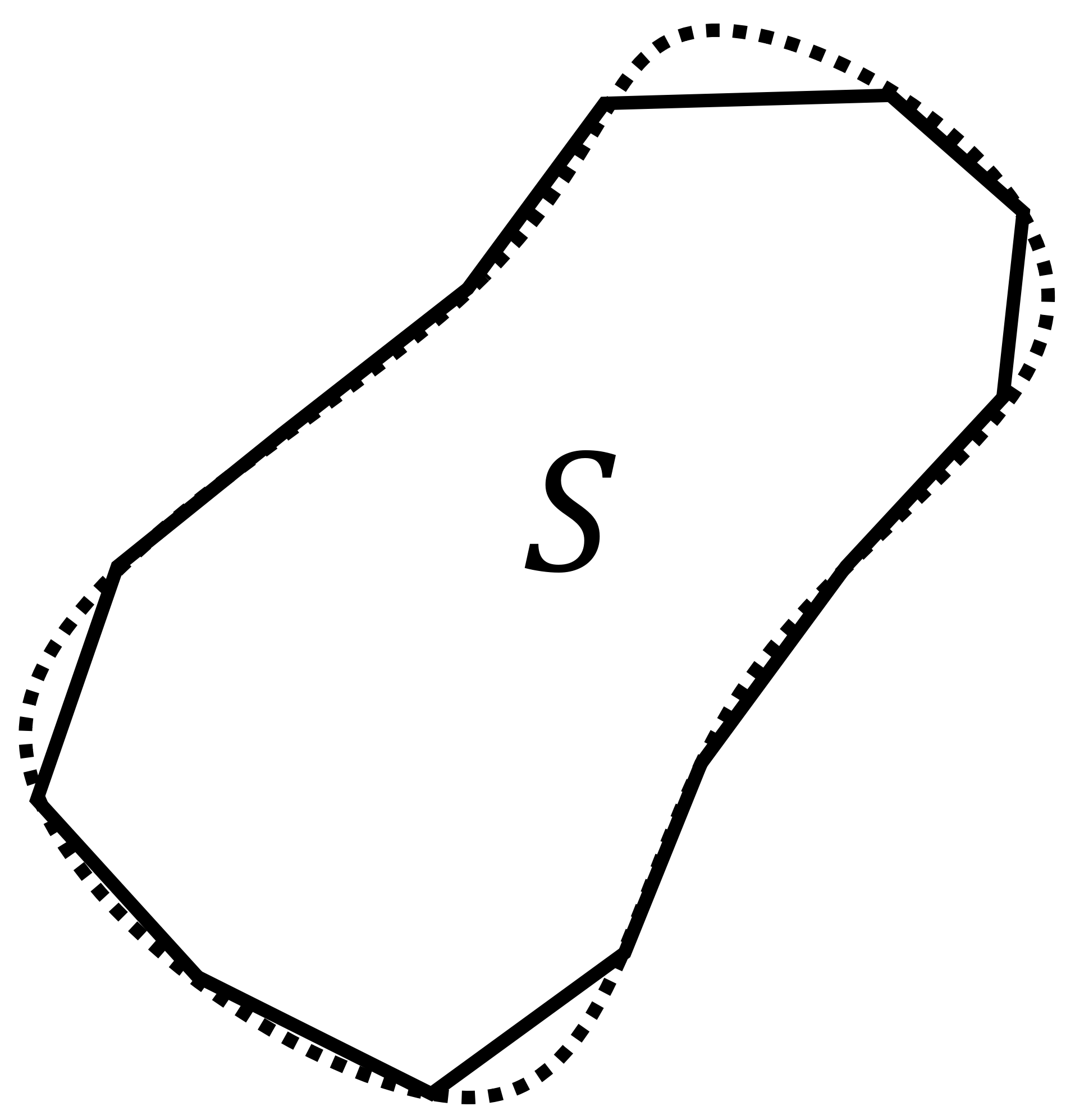}
		}
		\subfloat[Triangulation]
		{
			\label{fig23}\includegraphics[width=0.3\textwidth]{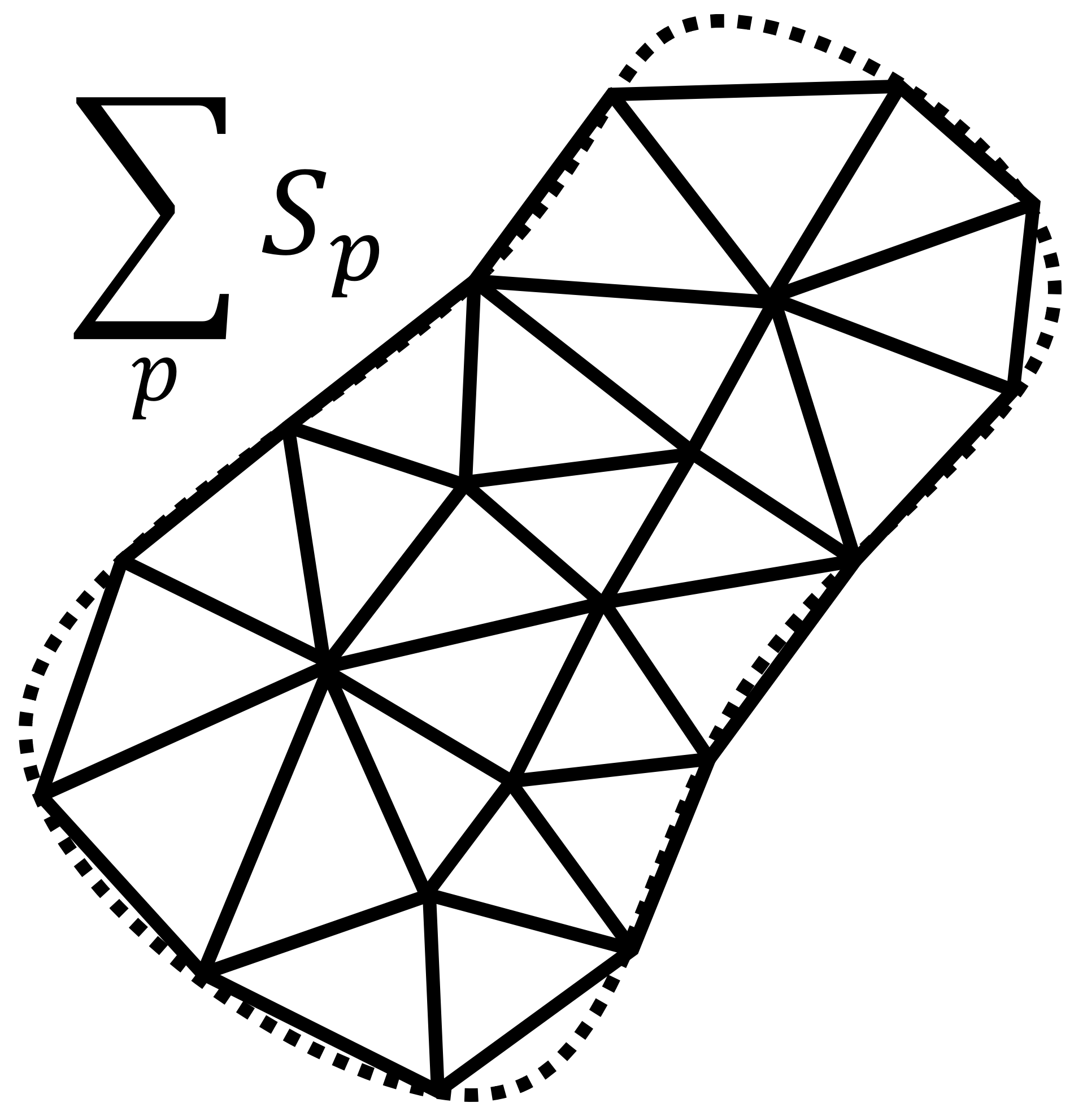}
		}
		\caption{Procedure of polygonal approximation and triangulation}    
		\label{polygon approximation and triangulation}
	\end{figure}
	In the procedure of polygonal approximation and triangulation to the domain $S_0$, the error is introduced by those triangles that have one edge adjacent to the curve boundary. We mark those triangles with $S_1,S_2,\cdots$ and the corresponding mismatching arch domain with $\Delta S_1,\Delta S_2,\cdots$, as shown in Figure \ref{Fig1}:
    \begin{figure}[H]
		\centering
		\includegraphics[width=0.4\textwidth]{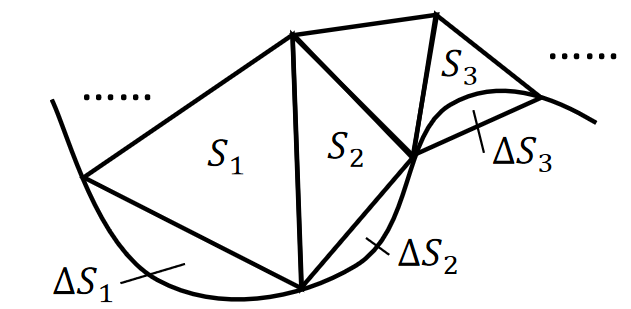}
		\caption{Error of polygon approximation} 
		\label{Fig1} 
	\end{figure}
	
	To estimate the error of polygon approximation, we first introduce a lemma from \cite{Numerical}.
	\begin{lemma}\label{interpolation}
		Let $f(x)$ have an $(n+1)$-th derivative, $f^{(n+1)}(x)$, in an interval $[a,b]$. Let $P_n(x)$ be the interpolation polynomial for f(x) with respect to $n+1$ distinct points $x_k,k=0,1,\cdots,n$ in the interval $[a,b]$(i.e. $P_n(x_k)=f(x_k)$ and $x_k\in[a,b]$). Then for each $x\in [a,b]$ there exists a point $\xi = \xi (x)$ in the open interval:
		\begin{equation*}
			\min(x_0,x_1,\ldots,x_n,x)<\xi<\max(x_0,x_1,\ldots,x_n,x),
		\end{equation*}
		such that
		\begin{equation*}
			f(x)-P_{n}(x)\equiv R_{n}(x) =\frac{(x-x_{0})(x-x_{1})\cdots(x-x_{n})}{(n+1)!}f^{(n+1)}(\xi).
		\end{equation*}
	\end{lemma}
	As for the condition we need, we choose $n=1, x_0=a, x_n=b$, and immediately we can get
	\begin{equation*}
		\min(a,b,x)<\xi<\max(a,b,x),
	\end{equation*}
	\begin{equation*}
		f(x)-P_1(x)= R_1(x)=\frac{(x-a)(x-b)}{2}f''(\xi).
	\end{equation*}
	Based on the error estimate of piecewise linear interpolation, in each small triangle on the boundary, we can express the curve boundary as $f_k(t)=x_k(t)\mathbf{i}+y_k(t)\mathbf{j}$, where each component function is $p-1$-th order continuous since the boundary is represented by B-spline functions. The linear interpolation function $P_1^k(t)=P_{1x}^k(t)\mathbf{i}+P_{1y}^k(t)\mathbf{j}$ is used to describe the triangle edge adjacent to the curve boundary.
	Based on the Lemma \ref{interpolation}, we give a theorem to estimate the error of discretization.
	\begin{theorem}\label{integral estimate}
		When we approximate the integral on the domain $S_0$
		\begin{equation}
			U(x_i,y_i)=\iint\limits_{S_0}H(x_i-x_o,y_i-y_o) dx_ody_o,\label{bigU}
		\end{equation}
		by the sum of integral on each triangulation
		\begin{equation}
			\tilde{U}(x_i,y_i)=\sum_{p=1}^{N_T}\iint\limits_{S_p} H(x_i-x_o,y_i-y_o) dx_ody_o,\label{tildeU}
		\end{equation}
		we have error estimate as follows:
		\begin{equation*}
			|U(x_i,y_i)-\tilde{U}(x_i,y_i)|\leqslant C\max h_k^2\rightarrow 0\quad as \quad \max h_k^2\rightarrow0,
		\end{equation*}
		where $h_k$ represents the length of triangle edge adjacent to the curve boundary and $C$ is a constant independent of $h_k$.
	\end{theorem}
	\begin{proof}
		According to Lemma \ref{interpolation}, we can get
		\begin{equation*}
			||f(x)-P_1(x)||_\infty= ||R_1(x)||_\infty\leqslant c_0h^2,
		\end{equation*}
		where $c_0=max\{f''(\xi)\}/2,~h$ is the length of the interval for interpolation.
		For one arch domain on the boundary, we can use $f_k(t)$ to describe the curve boundary and $P_1^k(t)$ to describe the corresponding linear approximation.
	So that we have estimation:
	\begin{equation*}
		|\Delta S_k|\leqslant h_k d_k,
	\end{equation*}
	where $\Delta S_k$ represents the $k$-th arch domain, $|\Delta S_k|$ is the measure of $\Delta S_k$, $h_k$ is the length of the $k$-th linear interpolation edge on the boundary, and $d_k$ is the height of the $k$-th arch. $d_k$ can be bounded by the interpolation error estimate
	\begin{equation*}
		d_k\leqslant \max\limits_{t}\left\{|x_k(t)-P_{1x}^k(t)|+|y_k(t)-P_{1y}^k(t)|\right\}\leqslant c_1h_k^2,
	\end{equation*}
    where $c_1$ is some constant independent of $h_k$. An illustration of these variables and the estimation is shown in Figure \ref{4_1}:
    \begin{figure}[H]
		\centering
		\includegraphics[width=0.4\textwidth]{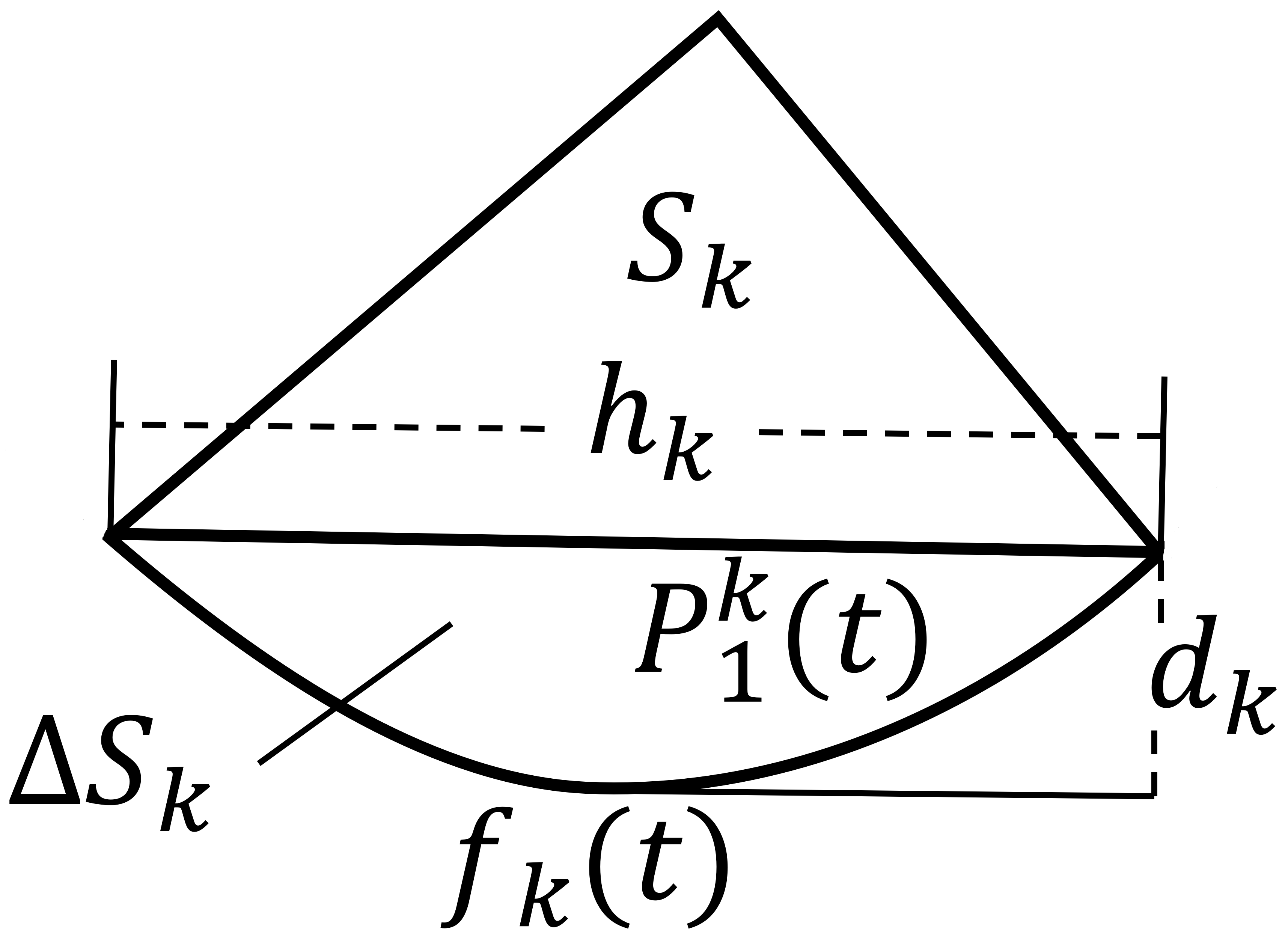}
		\caption{Error estimation in one arch} 
		\label{4_1} 
	\end{figure}

	So for the whole curve boundary, the upper bound for error of integral area can be written as
	\begin{equation}
		\sum_{k}|\Delta S_k|\leqslant \sum_{k}c_1h_k^3.\label{deltaS}
	\end{equation}
    
	Compare the two formulas \eqref{bigU}-\eqref{tildeU}, and consider that $\sum\limits_{k}h_k$ can be controlled by the perimeter of the curve boundary $C_{peri}$, we have the following estimation:
	\begin{equation*}
		\begin{aligned}
			|\Delta U(x_i,y_i)|&:=|U(x_i,y_i)-\tilde{U}(x_i,y_i)|\leqslant\sum\limits_{k}\left|\iint\limits_{\Delta S_k}H(x_i-x_o,y_i-y_o) dx_ody_o\right|\\
			&\leqslant\max |H(x_i-x_o,y_i-y_o)|~\sum\limits_{k}|\Delta S_k| \leqslant c_2\sum_{k}h_k^3\leqslant c_2\max\limits_{k} h_k^2\sum_{k}h_k\\
			&\leqslant c_2C_{peri}\max h_k^2:=C\max h_k^2\rightarrow 0\quad as \quad \max h_k^2\rightarrow0.
		\end{aligned}
	\end{equation*}
    Here we use \eqref{deltaS} in the second line of the above estimate and $c_2=c_1\max |H(x_i-x_o,y_i-y_o)|$.
   \end{proof} 
	With more vertices added to the boundary of the spline boundary and refinement of triangulation, it is easy to ensure that $\max h_k^2\rightarrow 0$, which means that the integral error caused by the polygonal approximation can be well controlled.

	\subsection{Numerical imaging formula}
	To numerically compute the integral on each triangle, we use the symmetric Gaussian quadrature rules provided in \cite{Dunavant}. Then with the help of the quadrature rules,  Eq. (\ref{U_triangle}) can be further approximated by
	\begin{equation}\label{numerical imaging}	U(x_i,y_i)\approx\sum_{p=1}^{N_T}\sum_{q=1}^{N_G}W_{G,q}H(x_i-x_o^{pq},y_i-y_o^{pq})|S_{p}|,\quad I(x_i,y_i)=|U(x_i,y_i)|^2,
	\end{equation}
	where $(x_o^{pq},y_o^{pq})$ is the coordinate of $q$-th Gaussian point in the $p$-th triangle in the triangulation of polygon $S$, $W_{G,q}$ is the weight of $q$-th Gaussian point. $S_{p}$ represents the $p$-th triangle, the measure of which is $|S_{p}|$. And $N_G$ is the total number of Gaussian points.
	
	If the $N_G$-point Gaussian quadrature formula used in this article has precision of degree $n$, which means the scheme can achieve exact result for all polynomials with degree less than or equal to $n$, the error estimation of the formula in one triangle can be represented as follows:
	\begin{equation}\label{integral error}
		\iint\limits_{S_{p}}H(x_i-x_o,y_i-y_o) dx_ody_o-\sum_{q=1}^{N_G}W_{G,q}H(x_i-x_o^{pq},y_i-y_o^{pq})|S_{p}|\leqslant Ch_p^{n+1},
	\end{equation}
	where $h_p$ is the length of the longest edge in the p-th triangle and $C$ is a constant independent of $h_p$. We should point out that for the polygon region $S$, if we assume that the family of triangulation is shape-regular quasi-uniform as defined in 
    \cite[Definition 3.2 and 3.3]{zmchen},   
    we can easily derive that the total number of triangles $N_T$ generated by triangulation has a quadratic relationship with the reciprocal of the maximum of edge length $h_p$, i.e. $N_T\sim O(1/\max h_p^2)$. Sum both sides of Eq. \eqref{integral error} over $p$ and we can get:
    \begin{equation}\label{errororder}
    \begin{aligned}
        \sum_{p=1}^{N_T}\iint\limits_{S_{p}}H(x_i-x_o,y_i-y_o) dx_ody_o&-\sum_{p=1}^{N_T}\sum_{q=1}^{N_G}W_{G,q}H(x_i-x_o^{pq},y_i-y_o^{pq})|S_{p}|\\
        \leqslant \sum_{p=1}^{N_T}Ch_p^{n+1}&\leqslant N_TC\max h_p^{n+1}\sim O(\max h_p^{n-1})
    \end{aligned}
    \end{equation}
    \begin{remark}\label{remark1}
        Theorem \ref{integral estimate} implies that the accuracy of polygon approximation is of 2nd order, and Eq. \eqref{errororder} indicates that the numerical integration formula has an accuracy of order $n-1$. To prevent the loss of accuracy, here the Gaussian quadrature formula we apply should have at least a precision of order 3. 
    \end{remark}
	
	However, the coordinates of Gaussian points cannot exist independently. They are determined by the coordinates of the vertex of the triangle in which they are located. Without loss of generality, we represent the three vertices in the $p$-th triangle by $R_{p,a},R_{p,b},R_{p,c}$. According to the areal coordinate representation of Gaussian points, for the $p$-th triangle, we can obtain the relation between vertex coordinates and Gaussian point coordinates:
	
	\begin{equation}\label{R_multiple_L}
			\mathbf{G}_p=\begin{pmatrix}
				x_o^{p1}&x_o^{p2}&\cdots&x_o^{pN_G}\\
				y_o^{p1}&y_o^{p2}&\cdots&y_o^{pN_G}	
			\end{pmatrix}
			=
			\begin{pmatrix}
				R_{p,ax}&R_{p,bx}&R_{p,cx}\\
				R_{p,ay}&R_{p,by}&R_{p,cy}		
			\end{pmatrix}
			\mathbf{L},
	\end{equation}
	where $\mathbf{G}_p\in\mathbb{R}^{2\times N_G}$ is the Gaussian point matrix in the $p$-th triangle. $\mathbf{L}\in\mathbb{R}^{3\times N_G}$ is the weight matrix of the integral points with respect to the vertices, each column of which determined the location of a Gaussian point.
	
	Then we consider how to build the relationship between the vertices of each triangle and the control points for periodic B-spline curve, so that we can explicitly compute the gradient later. For a given sequence of $n$ control points $P_1,P_2,\cdots,P_n$, we can generate a spline curve by these control points. Every point on the curve corresponds to a curve parameter. To discretize the curve and further approximate the region by polygon, we should arrange some points on the curve. Here we choose $m$ curve parameters $t_1,t_2,\cdots,t_m$ and obtain the discrete points $Q_1,Q_2,\cdots,Q_m$ on the curve boundary in the following way:
	\begin{equation}\label{P2Q}
		\begin{pmatrix}
			Q_1\\
			Q_2\\
			\vdots\\
			Q_m
		\end{pmatrix}=
		\begin{pmatrix}
			\mathring{N}_{1,p}(t_1)&\mathring{N}_{2,p}(t_1)&\cdots&\mathring{N}_{n,p}(t_1)\\
			\mathring{N}_{1,p}(t_2)&\mathring{N}_{2,p}(t_2)&\cdots&\mathring{N}_{n,p}(t_2)\\
			\vdots&\vdots&\ddots&\vdots\\
			\mathring{N}_{1,p}(t_m)&\mathring{N}_{2,p}(t_m)&\cdots&\mathring{N}_{n,p}(t_m)
		\end{pmatrix}
		\begin{pmatrix}
			P_1\\
			P_2\\
			\vdots\\
			P_n
		\end{pmatrix},
	\end{equation}
	where $t_k\in[0,1], k=1,2,\cdots,m$ and $t_k<t_{k+1},k=1,2,\cdots,m-1$. So that the coordinates of discrete points $Q_1,Q_2,...,Q_m$ can be represented by the linear combination of coordinates of control points. Now we can describe the process in matrix multiplication form:
	
	\begin{equation}\label{P2Qmatrix}
		\mathbf{Q}=\mathbf{NP},\quad \mathbf{Q}\in\mathbb{R}^{m\times2},\mathbf{N}\in\mathbb{R}^{m\times n},\mathbf{P}\in\mathbb{R}^{n\times 2}.
	\end{equation}
	
	The triangulation of domain $S_0$ is performed by Delaunay algorithm. To perform the triangulation, we should choose $m$ points $Q_1,Q_2,..., Q_m$ on the curve first to discretize the boundary curve. Then we perform Delaunay triangulation based on these points to obtain a rough triangular mesh, which does not have any inner vertices. After deleting triangles out of the boundary, the further triangulation of $S_0$ is performed by refinement of those initial rough meshes. 
    
    In the following, we consider the process of determining an inner triangle by the points $Q_1,Q_2,\cdots,Q_m$ on the spline curve.
    First, all the vertices of the triangles in the initial mesh are just $Q_1, Q_2,...,Q_m$, we rename them by $\tilde{R}_1,\tilde{R}_2,\cdots,\tilde{R}_m$. In matrix form we have:
	\begin{equation}\label{Q2R1}
		\begin{pmatrix}
			\tilde{R}_1\\
			\tilde{R}_2\\
			\vdots\\
			\tilde{R}_m
		\end{pmatrix}
		=
		\begin{pmatrix}
			1&&&\\
			&1&&\\
			&&\ddots&\\
			&&&1
		\end{pmatrix}
		\begin{pmatrix}
			Q_1\\
			Q_2\\
			\vdots\\
			Q_m
		\end{pmatrix}
		=\mathbf{I}_{m}
		\begin{pmatrix}
			Q_1\\
			Q_2\\
			\vdots\\
			Q_m
		\end{pmatrix}.
	\end{equation}
	Then, to refine the Delaunay triangulation, we need to add new points. For a triangle larger than our tolerance, whose three vertices are $\tilde{R}_{i_0},\tilde{R}_{j_0},\tilde{R}_{k_0}$, we add its center of gravity as a new point. If we call the new point $\tilde{R}_{m+1}$, then we can modify the link matrix between $\{\tilde{R}_k\}_{k=1}^{m+1}$ and $\{Q_k\}_{k=1}^{m}$ in the following way:
	\begin{equation}\label{Q2R2}
		\begin{pmatrix}
			\tilde{R}_1\\
			\tilde{R}_2\\
			\vdots\\
			\tilde{R}_{m+1}
		\end{pmatrix}
		=
		\begin{pmatrix}
			\mathbf{I}_m\\
			\frac{1}{3}(\mathbf{e}_{i_0}+\mathbf{e}_{j_0}+\mathbf{e}_{k_0})
		\end{pmatrix}
		\begin{pmatrix}
			\tilde{R}_1\\
			\tilde{R}_2\\
			\vdots\\
			\tilde{R}_{m}
		\end{pmatrix}
		=\mathbf{W}_{m+1,m}
		\begin{pmatrix}
			\tilde{R}_1\\
			\tilde{R}_2\\
			\vdots\\
			\tilde{R}_{m}
		\end{pmatrix}
		=\mathbf{W}_{m+1,m}\mathbf{I}_m
		\begin{pmatrix}
			Q_1\\
			Q_2\\
			\vdots\\
			Q_m
		\end{pmatrix},
	\end{equation}
	where $\mathbf{e}_{i_0}\in\mathbb{R}^{1\times m}$ is a row vector in which only the $i_0$-th element is $1$ and the other elements are $0$. After adding the new point, we can refine the triangulation. Repeat the process until every triangle is smaller than our tolerance. In this process, every new point added can be represented by the average of three vertices of the triangle located. Assuming that there are $K$ points in the final triangulation, we can get a series of link matrix $\mathbf{W}_{m+k+1,m+k}(k=0,1,\cdots,K-m-1)$, the definition of which is:
	\begin{equation}\label{Q2R3}
		\mathbf{W}_{m+k+1,m+k}=
		\begin{pmatrix}
			\mathbf{I}_{m+k}\\
			\frac{1}{3}(\mathbf{e}_{i_0}+\mathbf{e}_{j_0}+\mathbf{e}_{k_0})
		\end{pmatrix},
	\end{equation}
	where $i_0,j_0,k_0$ are the indexes of the points forming the triangle which needs to be refined in the $k$-th adding points.
	
	In this way, every vertex in the final triangulation can be represented by linear combinations of the points $Q_1,Q_2,\cdots,Q_m$ by the following formula:
	\begin{equation}\label{Q2R4}
		\begin{pmatrix}
			\tilde{R}_1\\
			\tilde{R}_2\\
			\vdots\\
			\tilde{R}_{K}
		\end{pmatrix}
		=
		\mathbf{W}_{K,K-1}\mathbf{W}_{K-1,K-2}\cdots\mathbf{W}_{m+1,m}\mathbf{I}_m
		\begin{pmatrix}
			Q_1\\
			Q_2\\
			\vdots\\
			Q_m
		\end{pmatrix}
		=\mathbf{W}\begin{pmatrix}
			Q_1\\
			Q_2\\
			\vdots\\
			Q_m
		\end{pmatrix},
	\end{equation}
	with the definition:
	\begin{equation}\label{Q2R5}
		\mathbf{W}=\mathbf{W}_{K,K-1}\mathbf{W}_{K-1,K-2}\cdots\mathbf{W}_{m+1,m}\mathbf{I}_m.
	\end{equation}
	
	As a result, every vertex of each triagnle can be represented in the form of matrix multiplication:
	\begin{equation}\label{Q2R6}
		\tilde{\mathbf{R}}=\mathbf{WQ},\quad \tilde{\mathbf{R}}\in\mathbb{R}^{K\times2}, \mathbf{W}\in\mathbb{R}^{K\times m},\mathbf{Q}\in\mathbb{R}^{m\times 2},
	\end{equation}
	where $\mathbf{W}$ is the coefficient of linear combination.
	
	After refining the Delaunay triangulation, we will obtain all point coordinates $\tilde{\mathbf{R}}$ and connectivity list $\mathbf{C}\in\mathbb{R}^{N_T\times 3}$, each row of $\mathbf{C}$ consists of indexes of the three vertexes of a single triangle. Then we will reshape the matrix $\tilde{\mathbf{R}}$ into $\mathbf{R}$ according to the connectivity list. An intuitive structure is:
	\begin{equation}\label{reshape1}
		\mathbf{R}(:,:,1)=
		\begin{pmatrix}
			R_{1,ax}&R_{1,bx}&R_{1,cx}\\
			R_{2,ax}&R_{2,bx}&R_{2,cx}\\
			\vdots&\vdots&\vdots\\
			R_{N_T,ax}&R_{N_T,bx}&R_{N_T,cx}\\
		\end{pmatrix},
		\mathbf{R}(:,:,2)=
		\begin{pmatrix}
			R_{1,ay}&R_{1,by}&R_{1,cy}\\
			R_{2,ay}&R_{2,by}&R_{2,cy}\\
			\vdots&\vdots&\vdots\\
			R_{N_T,ay}&R_{N_T,by}&R_{N_T,cy}\\
		\end{pmatrix}.
	\end{equation}
	$\mathbf{R}$ is a three-dimensional matrix, written $\mathbf{R}\in\mathbb{R}^{N_T\times3\times2}$. $\mathbf{R}(:,:,1)$ stores the X-coordinate information for each point, and $\mathbf{R}(:,:,2)$ stores the Y-coordinate information for each point. Each row includes the information of three vertices of a single triangle. An easy way to carry out the reshaping process is to assign values in $\tilde{\mathbf{R}}$ by the indexes in $\mathbf{C}$. We note the procedure by 
	\begin{equation}\label{reshape2}
		\mathbf{R}=\tilde{\mathbf{R}}(\mathbf{C}), \mathbf{R}\in\mathbb{R}^{N_T\times3\times2}, \tilde{\mathbf{R}}\in\mathbb{R}^{K\times2}, \mathbf{C}\in\mathbb{R}^{N_T\times 3}.
	\end{equation}
	
	As mentioned above, to obtain the coordinates of Gaussian points, we only need to multiply a weight matrix connecting the vertices and the Gaussian points $\mathbf{L}$ to the right, which makes our calculations a lot easier. And finally, we have successfully established the process between Gaussian points and control points. An illustration of the process is shown in Figure \ref{PQR}.
	\begin{figure}[H]
		\centering
		\includegraphics[width=0.4\textwidth]{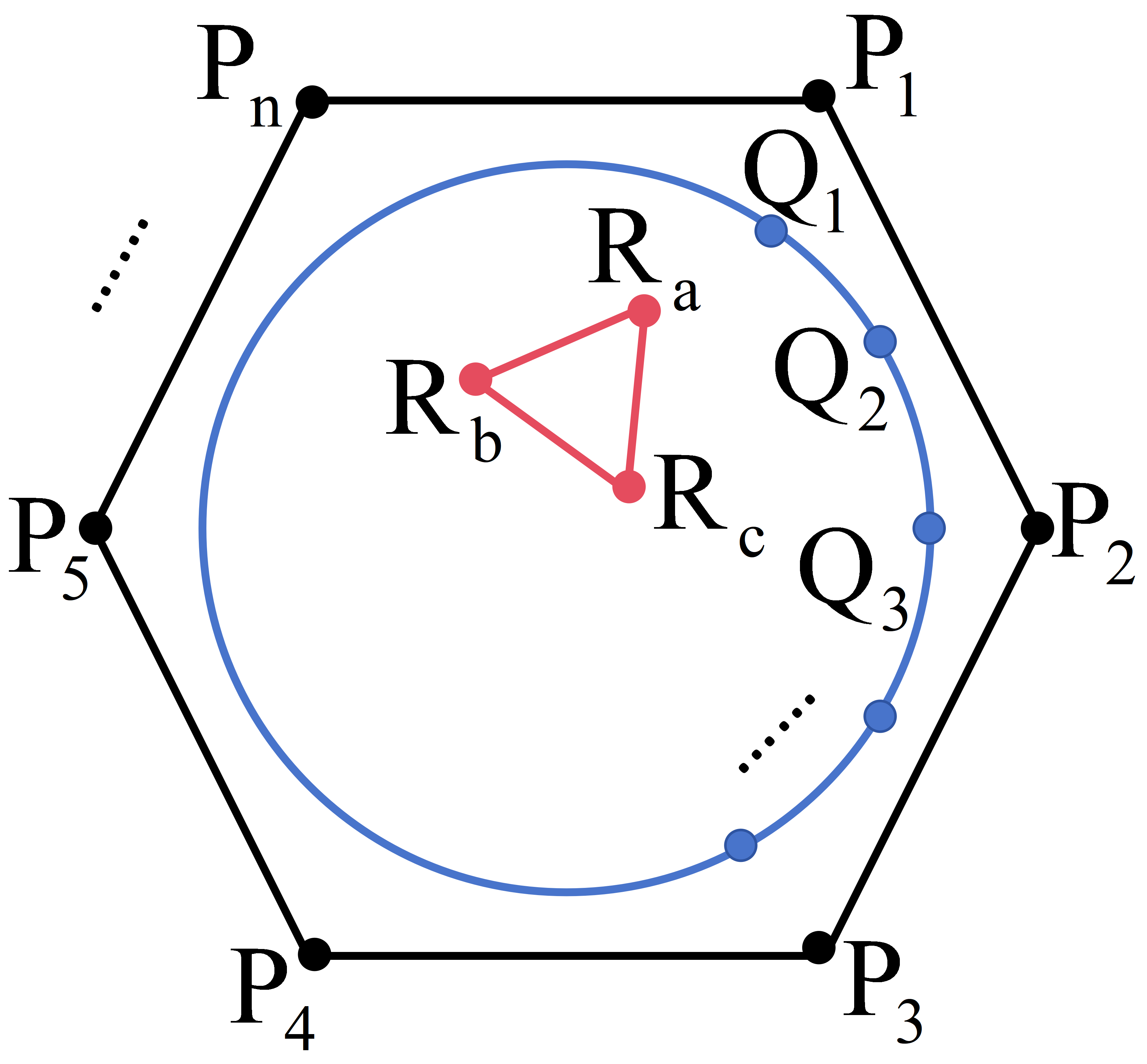}
		\caption{explicit relation building} 
		\label{PQR} 
	\end{figure}
	
	For the sampling points on the image plane, we usually choose them equidistantly to simplify the following data visualization. In this way, we finally finished the numerical imaging formula derivation.
	\subsection{Inverse lithography problem}
    
	After finishing the derivation of imaging formula and obtaining the explicit relationship between the control points and the Gaussian points, we are in a position to discuss the inverse lithography problem. We remark that normalization of coordinates has no essential effect on the inverse problem. So in this part, we will still omit the $\hat{ }$ symbol. 
	
	First, we will introduce the optimization model of the inverse lithography problem. We note the intensity distribution function obtained by numerical computation as $I(x_i,y_i)=|U(x_i,y_i)|^2=U(x_i,y_i)*\overline{U(x_i,y_i)}$, where $(x_i,y_i)$ represents the coordinate on the image plane. Taking into account the effects of photoresist on the wafer, in the optical lithography simulation, a widely used model is to truncate the intensity by a threshold function:
	\begin{equation*}
		T(x)=
		\begin{cases}
			1\quad x\geqslant tr,\\
			0\quad x<tr,
		\end{cases}
	\end{equation*}
	where $tr$ is the threshold. However, the hard truncation function is discontinuous, which leads to the difficulty of calculating its gradient. So we will approximate it by sigmoid function:
	\begin{equation*}
		sig(x)=\frac{1}{1+exp(-a(x-tr))},
	\end{equation*}
	where $a$ controls the approximation accuracy. Given the target intensity distribution function $I_{target}(x_i,y_i)$ in a binary way, then the mask optimization problem can be reduced to an unconstrained optimization problem:
	\begin{equation*}
		\min\limits_{P_1,\cdots,P_n} J:=||sig(I)-I_{target}||^2_2=\iint_{\Omega}(sig(I)-I_{target})^2dxdy,
	\end{equation*}
	where domain $\Omega$ is the whole image plane. 
	Further we will calculate the derivative of the objective function with respect to the horizontal and vertical coordinates $P_{1x},P_{2x},\cdots,P_{nx}$ and $P_{1y},P_{2y},\cdots,P_{ny}$ of the control points $P_1,P_2,\cdots,P_n$:
	\begin{equation*}
		\left(
		\begin{aligned}
			\frac{\partial J}{\partial P_{kx}}\\
			\frac{\partial J}{\partial P_{ky}}
		\end{aligned}
		\right)
		=
		\left(
		\begin{aligned}
			\iint\limits_{\Omega}2(sig(U\overline{U})-I_{target})sig'(U\overline{U})(\overline{U}\frac{\partial U}{\partial  P_{kx}}+U\frac{\partial\overline{U}}{\partial  P_{kx}})dx_idy_i\\
			\iint\limits_{\Omega}2(sig(U\overline{U})-I_{target})sig'(U\overline{U})(\overline{U}\frac{\partial U}{\partial P_{ky}}+U\frac{\partial\overline{U}}{\partial P_{ky}})dx_idy_i
		\end{aligned}
		\right), \quad k=1,2,\cdots,n.
	\end{equation*}
	The integral is carried out in the image plane. Now we discretize the integrals equidistantly on the image plane by $N_x\times N_y$ sampling points $(x_i^m,y_i^n), m=1,2,\cdots,N_x, n=1,2,\cdots,N_y$:
    \begin{equation}\label{grad1}
		\begin{aligned}
            \sum\limits_{m,n=1}^{N_{x},N_{y}}2\left(sig(U(x_i^m,y_i^n)\overline{U(x_i^m,y_i^n)})-I_{target}(x_i^m,y_i^n))sig'(U(x_i^m,y_i^n)\overline{U(x_i^m,y_i^n)}\right)\\
			\left(\overline{U(x_i^m,y_i^n)}\frac{\partial U(x_i^m,y_i^n)}{\partial  P_{kx}}+U(x_i^m,y_i^n)\frac{\partial\overline{U(x_i^m,y_i^n)}}{\partial P_{kx}}\right)\Delta x_i\Delta y_i,\\
            \sum\limits_{m,n=1}^{N_{x},N_{y}}2\left(sig(U(x_i^m,y_i^n)\overline{U(x_i^m,y_i^n)})-I_{target}(x_i^m,y_i^n))sig'(U(x_i^m,y_i^n)\overline{U(x_i^m,y_i^n)}\right)\\
			\left(\overline{U(x_i^m,y_i^n)}\frac{\partial U(x_i^m,y_i^n)}{\partial  P_{ky}}+U(x_i^m,y_i^n)\frac{\partial\overline{U(x_i^m,y_i^n)}}{\partial  P_{ky}}\right)\Delta x_i\Delta y_i.
		\end{aligned}
    \end{equation}
		Considering that conjugated functions can be directly transformed into each other, the problem can be reduced to calculating the explicit expression of $\frac{\partial U}{\partial P_{kx}}$ and $\frac{\partial U}{\partial P_{ky}}$.
	
	We come back to the formula Eq. \ref{numerical imaging} that computes the intensity on the image plane:
	\begin{equation}\label{grad2}
		U(x_i,y_i)=\sum_{p=1}^{N_T}\sum_{q=1}^{N_G}W_{G,q}H(x_i-x_o^{pq},y_i-y_o^{pq})|S_{p}|,\quad I(x_i,y_i)=|U(x_i,y_i)|^2.
	\end{equation}
	In this formula, $H$ and $|S_{p}|$ are related to the coordinates of the control points. The expression of derivative can be written as:
	\begin{equation}\label{grad3}
    \begin{aligned}
        \frac{\partial U(x_i^m,y_i^n)}{\partial P_{kx}}=\sum_{p=1}^{N_T}\sum_{q=1}^{N_G}W_{G,q}\left[\frac{\partial H(x_i^m-x_o^{pq},y_i^n-y_o^{pq})}{\partial P_{kx}}|S_{p}|+H(x_i^m-x_o^{pq}, y_i^n-y_o^{pq})\frac{\partial |S_{p}|}{\partial P_{kx}}\right],\\
        \frac{\partial U(x_i^m,y_i^n)}{\partial P_{ky}}=\sum_{p=1}^{N_T}\sum_{q=1}^{N_G}W_{G,q}\left[\frac{\partial H(x_i^m-x_o^{pq},y_i^n-y_o^{pq})}{\partial P_{ky}}|S_{p}|+H(x_i^m-x_o^{pq}, y_i^n-y_o^{pq})\frac{\partial |S_{p}|}{\partial P_{ky}}\right].
    \end{aligned}
	\end{equation}
	The problem can be further reduced to the calculation of the following four derivatives:
	\begin{equation*}
		\frac{\partial |S_{p}|}{\partial P_{kx}},\quad
		\frac{\partial |S_{p}|}{\partial P_{ky}},\quad
		\frac{\partial H(x_i-x_o^{pq},y_i-y_o^{pq})}{\partial P_{kx}},\quad
		\frac{\partial H(x_i-x_o^{pq},y_i-y_o^{pq})}{\partial P_{ky}}.
	\end{equation*}
	
	We start from the calculation of $\frac{\partial |S_{p}|}{\partial P_{kx}},\frac{\partial |S_{p}|}{\partial P_{ky}}$. If we assume that the three vertexes of the $p$-th triangle are $R_{p,a},R_{p,b},R_{p,c}$ and they have been arranged in counterclockwise order, then the absolute value operation in the area formula can be omitted and it becomes:
	\begin{equation}\label{grad4}
		|S_{p}|=\frac{1}{2}((R_{p,bx}-R_{p,ax})(R_{p,cy}-R_{p,ay})-(R_{p,by}-R_{p,ay})(R_{p,cx}-R_{p,ax})).
	\end{equation}
	In this way, the derivatives of $|S_p|$ become:
	\begin{equation}\label{grad5}
    \begin{aligned}
        \frac{\partial |S_{p}|}{\partial P_{kx}}=\frac{1}{2}((\frac{\partial R_{p,bx}}{\partial P_{kx}}-\frac{\partial R_{p,ax}}{\partial P_{kx}})(R_{p,cy}-R_{p,ay})-(R_{p,by}-R_{p,ay})(\frac{\partial R_{p,cx}}{\partial P_{kx}}-\frac{\partial R_{p,ax}}{\partial P_{kx}})),\\
        \frac{\partial |S_{p}|}{\partial P_{ky}}=\frac{1}{2}((R_{p,bx}-R_{p,ax})(\frac{\partial R_{p,cy}}{\partial P_{ky}}-\frac{\partial R_{p,ay}}{\partial P_{ky}})-(\frac{\partial R_{p,by}}{\partial P_{ky}}-\frac{\partial R_{p,ay}}{\partial P_{ky}})(R_{p,cx}-R_{p,ax})).
    \end{aligned}
	\end{equation}
	There are six derivatives connecting the vertexes and the control points. Here we illustrate the calculation of derivatives using $\frac{\partial R_{p,ax}}{\partial P_{kx}}$ and $\frac{\partial R_{p,ay}}{\partial P_{ky}}$ as an example. 
	
	As mentioned above, in the process of generating matrix $\mathbf{R}$, we will obtain $\tilde{\mathbf{R}}=\mathbf{WNP}$ and connectivity list matrix $\mathbf{C}$. $R_{p,ax}$ is the $(p,1,1)$-th entry in $\mathbf{R}$, which comes from the $(\mathbf{C}(p,1),1)$-th entry in $\tilde{\mathbf{R}}$. And $R_{p,ay}$ is the $(p,1,2)$-th entry in $\mathbf{R}$, which comes from the $(\mathbf{C}(p,1),2)$-th entry in $\tilde{\mathbf{R}}$. If we note $\mathbf{T}=\mathbf{WN}\in\mathbb{R}^{K\times n}$, then we can get:
	\begin{equation}\label{grad6}
		\frac{\partial R_{p,ax}}{\partial P_{kx}}=\frac{\partial R_{p,ay}}{\partial P_{ky}}=\mathbf{T}(\mathbf{C}(p,1),k), \quad k=1,2,\cdots,n.
	\end{equation}
	Similarly, we can obtain:
	\begin{equation}\label{grad7}
		\frac{\partial R_{p,bx}}{\partial P_{kx}}=\frac{\partial R_{p,by}}{\partial P_{ky}}=\mathbf{T}(\mathbf{C}(p,2),k),\quad \frac{\partial R_{p,cx}}{\partial P_{kx}}=\frac{\partial R_{p,cy}}{\partial P_{ky}}=\mathbf{T}(\mathbf{C}(p,3),k), \quad k=1,2,\cdots,n.
	\end{equation}
	Now we have finished the calculation of $\frac{\partial |S_{p}|}{\partial P_{kx}},\frac{\partial |S_{p}|}{\partial P_{ky}}$. Then we consider the calculation of the last two derivatives. In the preliminary section, we have already obtained the formula of optical system transfer function in Eq.(\ref{H}). Here we also omitted the $\hat{}$ symbol that represents the normalization and we get:
	\begin{equation}\label{grad8}
		H(x_i-x_o,y_i-y_o)=\frac{J_1(2\pi\rho)}{\rho}=\frac{J_1(2\pi\sqrt{(x_i-x_o)^2+(y_i-y_o)^2})}{\sqrt{(x_i-x_o)^2+(y_i-y_o)^2})}.
	\end{equation}
	Then the two derivatives can be directly written as:
	\begin{equation}\label{grad9}
    \begin{aligned}
        \frac{\partial H(x_i-x_o^{pq},y_i-y_o^{pq})}{\partial P_{kx}}=\left[\frac{\partial J_1(2\pi\rho)}{\partial \rho}\frac{\partial\rho}{\partial P_{kx}}\rho-J_1(2\pi\rho)\frac{\partial\rho}{\partial P_{kx}}\right]\Big/\rho^2,\\
        \frac{\partial H(x_i-x_o^{pq},y_i-y_o^{pq})}{\partial P_{ky}}=\left[\frac{\partial J_1(2\pi\rho)}{\partial \rho}\frac{\partial\rho}{\partial P_{ky}}\rho-J_1(2\pi\rho)\frac{\partial\rho}{\partial P_{ky}}\right]\Big/\rho^2.
    \end{aligned}	
	\end{equation}
	In the formula, terms that we need to further explain are:
	\begin{equation*}
		\frac{\partial J_1(2\pi\rho)}{\partial \rho},\quad
		\frac{\partial\rho}{\partial P_{kx}},\quad
		\frac{\partial\rho}{\partial P_{ky}}.
	\end{equation*}
	Using the recursion property of the first kind of Bessel function we can get:
	\begin{equation*}
		J_{n-1}(x)-J_{n+1}(x)=2J_n'(x).
	\end{equation*}
	Then the first term can be directly calculated by:
	\begin{equation}\label{grad10}
		\frac{\partial J_1(2\pi\rho)}{\partial \rho}=2\pi J_1'(2\pi\rho)=\pi \left[J_{0}(2\pi\rho)-J_{2}(2\pi\rho)\right].
	\end{equation}
	As for the latter two terms, we can get:
	\begin{equation}\label{grad11}
		\frac{\partial\rho}{\partial P_{kx}}=\frac{x_o^{pq}-x_i^m}{\rho}\frac{\partial x_o^{pq}}{\partial P_{kx}},\quad
		\frac{\partial\rho}{\partial P_{ky}}=\frac{y_o^{pq}-y_i^n}{\rho}\frac{\partial y_o^{pq}}{\partial P_{ky}}.
	\end{equation}
	Recall Eq. (\ref{R_multiple_L}), in which we have:
	\begin{equation}\label{grad12}
		\begin{pmatrix}
			x_o^{p1}&x_o^{p2}&\cdots&x_o^{pN_G}\\
			y_o^{p1}&y_o^{p2}&\cdots&y_o^{pN_G}	
		\end{pmatrix}
		=
		\begin{pmatrix}
			R_{p,ax}&R_{p,bx}&R_{p,cx}\\
			R_{p,ay}&R_{p,by}&R_{p,cy}		
		\end{pmatrix}
		\mathbf{L}.
	\end{equation}
	It can be derived directly that:
	\begin{equation}\label{grad13}
		\left(
		\begin{aligned}
			\frac{\partial x_o^{p1}}{\partial P_{kx}}\quad\frac{\partial x_o^{p2}}{\partial P_{kx}}\quad\cdots\quad\frac{\partial x_o^{pN_G}}{\partial P_{kx}}\\
			\frac{\partial y_o^{p1}}{\partial P_{kx}}\quad\frac{\partial y_o^{p2}}{\partial P_{kx}}\quad\cdots\quad\frac{\partial y_o^{pN_G}}{\partial P_{kx}}	
		\end{aligned}
		\right)
		=
		\left(
		\begin{aligned}
			\frac{\partial R_{p,ax}}{\partial P_{kx}}\quad\frac{\partial R_{p,bx}}{\partial P_{kx}}\quad\frac{\partial R_{p,cx}}{\partial P_{kx}}\\
			\frac{\partial R_{p,ay}}{\partial P_{ky}}\quad\frac{\partial R_{p,by}}{\partial P_{ky}}\quad\frac{\partial R_{p,cy}}{\partial P_{ky}}		
		\end{aligned}
		\right)
		\mathbf{L}.
	\end{equation}
	
	Now we have obtained all the derivatives that need to be calculated, and finally obtained the gradient of the original optimization problem. The method can be extended to multiple separate regions. The only point need to be considered is that the amplitude $U$ should be the sum of amplitude produced by each region, but the derivatives $\frac{\partial U}{\partial P_{kx}},\frac{\partial U}{\partial P_{ky}}$ are only influenced by the region that the control points belong to.

	\section{Algorithm}\label{algorithm}
    With the information of $\frac{\partial J}{\partial P_{kx}},\frac{\partial J}{\partial P_{ky}},k=1,2,\cdots,n$, in fact we have obtained the gradient information $\nabla J$ of target function $J$ and so that we can solve the optimization problem by various gradient-based methods, e.g., steepest descent method and conjugate gradient method. In this article we apply steepest descent(SD) method to set the descent direction and golden section method to set the step size $\alpha_k$. The basic iteration formula is:
    \begin{equation*}
        \mathbf{P}_{k+1}=\mathbf{P}_k-\alpha_k\nabla J(\mathbf{P}_k).
    \end{equation*}
    One thing needs to be explained in particular is that, in the process of optimization, the shape of the domain could twist violently, and the mesh will probably intersect. So after every step, we should regenerate the mesh and corresponding matrices. Based on these ideas, we introduced the procedure of our method in Algorithm 1
	\begin{algorithm}[htbp]
		\caption{Inverse lithography based on periodic B-splines}\label{alg:cap}
		\begin{algorithmic}
			\Require ~~\\ 
			Maximum number of iterations $I_M$\\ Accuracy level $\epsilon$\\
			Wave length $\lambda_0$\\
			Numerical aperture $NA$\\
			Lateral magnification $M$\\
			Initial control points $\mathbf{P}_0$\\
			Target mask pattern $I_{target}$\\
			Curve parameter $t_1,t_2,\cdots,t_m$
			\Ensure Optimized control points $\mathbf{P}$
			\State normalize the control points $\mathbf{P}_0$ by Eq.(\ref{normalize})
			\State $\mathbf{P}\gets\mathbf{P}_0$
			\State generate discrete points on the periodic B-spline curve by Eq.(\ref{P2Q}) and obtain $\mathbf{Q}$, $\mathbf{N}$ defined in Eq.(\ref{P2Qmatrix})
			\State Delaunay triangulation and refinement to obtain $\tilde{\mathbf{R}},\mathbf{C}$ and $\mathbf{W}$ by Eq.(\ref{Q2R1}-\ref{Q2R6})
			\State reshape $\tilde{\mathbf{R}}$ into $\mathbf{R}$ by Eq.(\ref{reshape1},\ref{reshape2})
			\State $k\gets1$
			\While{$J=||sig(I_{real})-I_{target}||^2_2>\epsilon $ \&\& $k<I_M$}
			\State $k\gets k+1$
			\State compute $\nabla J$ by Eq. (\ref{grad1}-\ref{grad13})
			\State determine the optimal step size $\alpha_k$
			\If{$\alpha_k<\epsilon$}
			\State inverse normalize $\mathbf{P}$ by Eq.(\ref{inverse normalization})
			\State return $\mathbf{P}$
			\ElsIf{$\alpha_k\geq\epsilon$}
			\State $\mathbf{P}\gets\mathbf{P}-\alpha_k\nabla J$
			\State regenerate $\mathbf{N},\mathbf{Q},\mathbf{C},\mathbf{W},\tilde{\mathbf{R}},\mathbf{R}$
			\EndIf
			\EndWhile
			\State inverse normalization $\mathbf{P}$ by Eq.(\ref{inverse normalization})
			\State return $\mathbf{P}$ 
		\end{algorithmic}
	\end{algorithm}

\section{Numerical Results}\label{numerical}
In this section, we present some numerical examples for a variety of mask patterns to show the efficiency of our proposed algorithm. We implement Algorithm \ref{alg:cap} by Matlab R2024b. 
During the numerical calculation, the common parameters that we selected are $a=90, tr=0.3, M=-1, \lambda_0=193nm, NA=0.93$. In order to ensure that the boundary has a certain smoothness, the order of periodic B-spline we selected is $p=3$. The stop criterion is set by $\epsilon=1e-4$ and $I_M=100$. As we highlighted in Remark \ref{remark1}, here we choose a 4 point triangular Gaussian quadrature formula with precision of degree 3. In order to see the difference between the target and the output more clearly, we provide the information of edge placement error(EPE) in each example.

	\begin{figure}[htbp]
		\centering
		
		\begin{subfigure}[b]{0.32\textwidth}
			\centering
			\includegraphics[width=\textwidth]{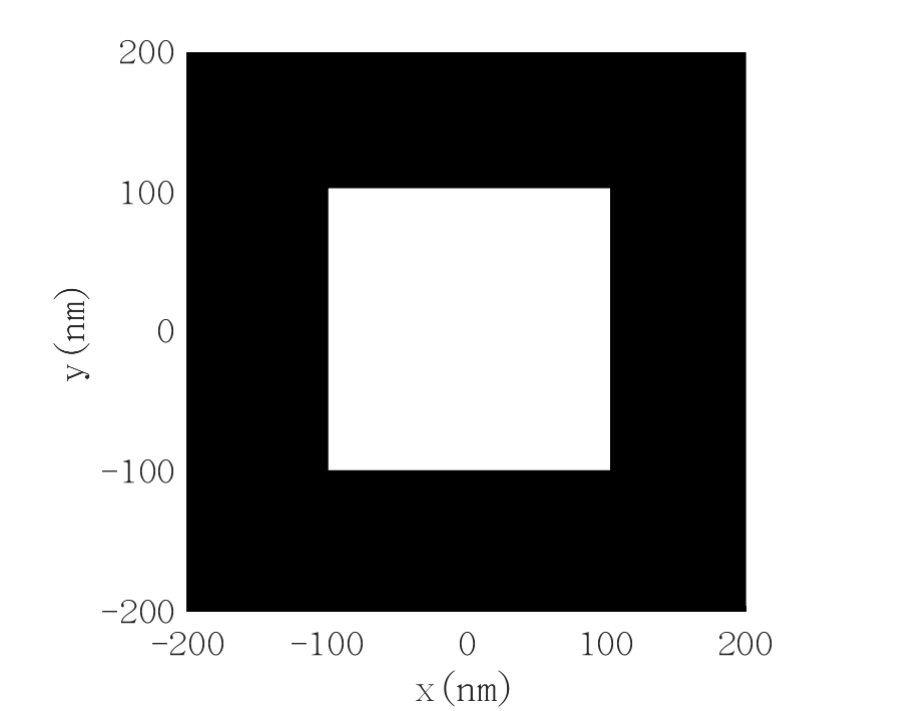}
			\caption{target pattern}
\label{fig1_1}
		\end{subfigure}
		\begin{subfigure}[b]{0.32\textwidth}
			\centering
			\includegraphics[width=\textwidth]{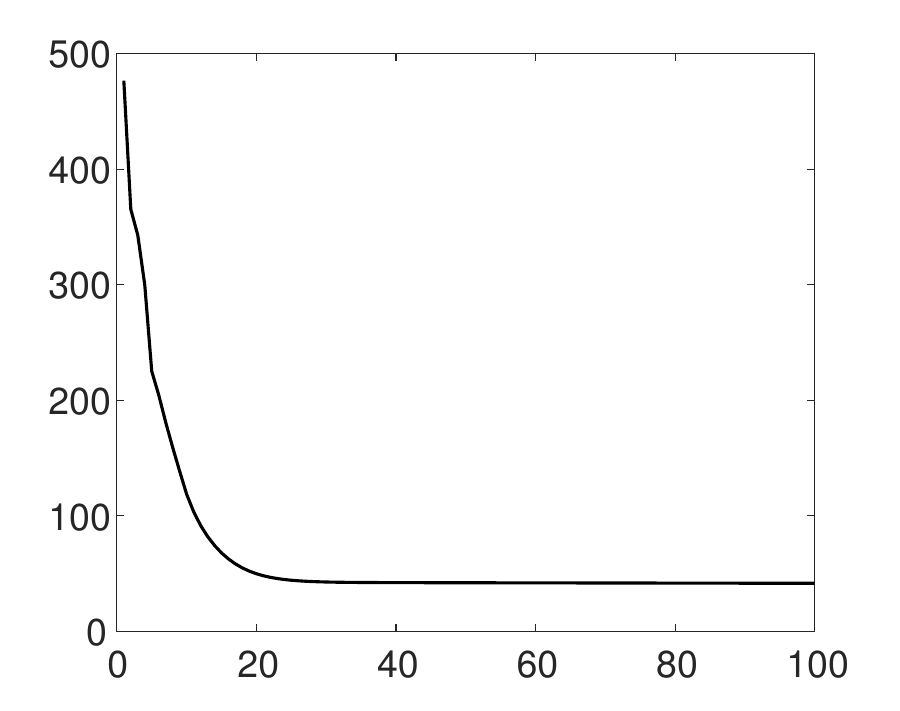}
			\caption{error decreasing}
			\label{fig1_2}
		\end{subfigure}
		
		\begin{subfigure}[b]{0.32\textwidth}
			\centering
			\includegraphics[width=\textwidth]{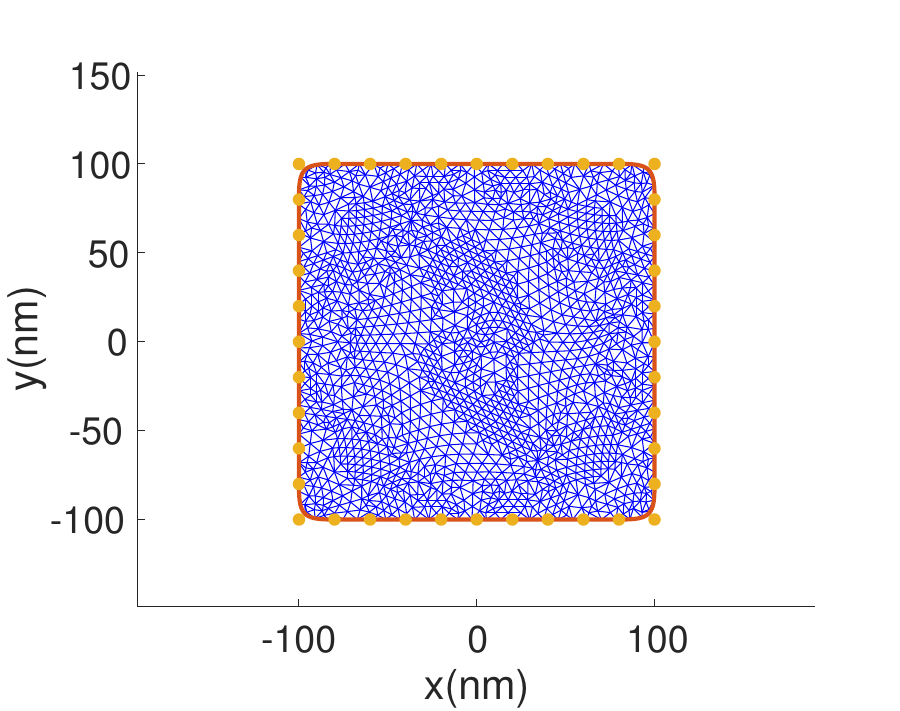}
			\caption{initial mask}
			\label{fig1_3}
		\end{subfigure}
		\begin{subfigure}[b]{0.32\textwidth}
			\centering
			\includegraphics[width=\textwidth]{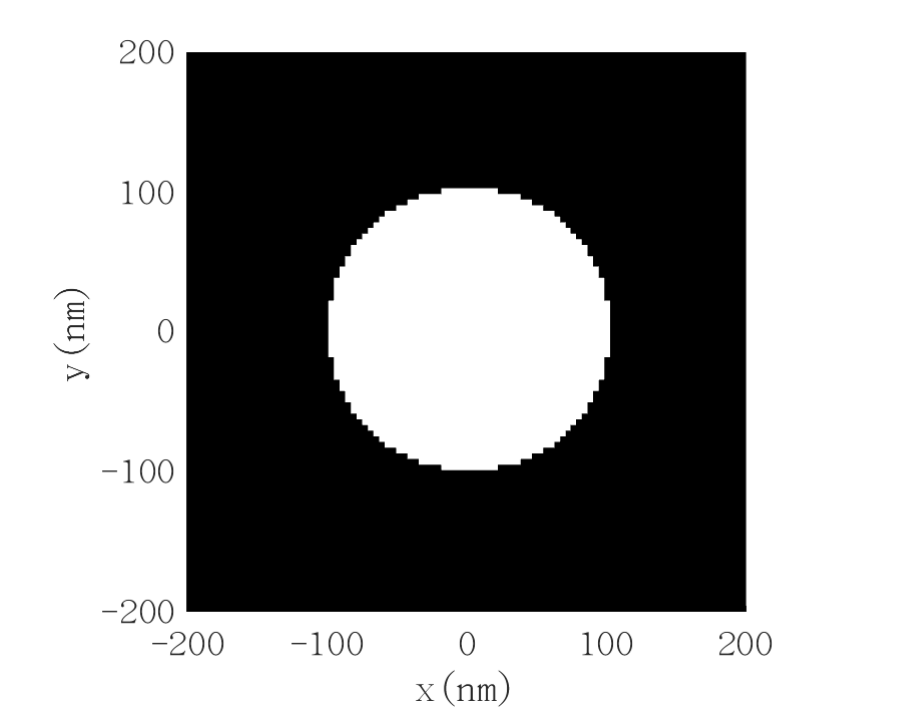}
			\caption{output of initial mask}
			\label{fig1_4}
		\end{subfigure}
		\begin{subfigure}[b]{0.32\textwidth}
			\centering
			\includegraphics[width=\textwidth]{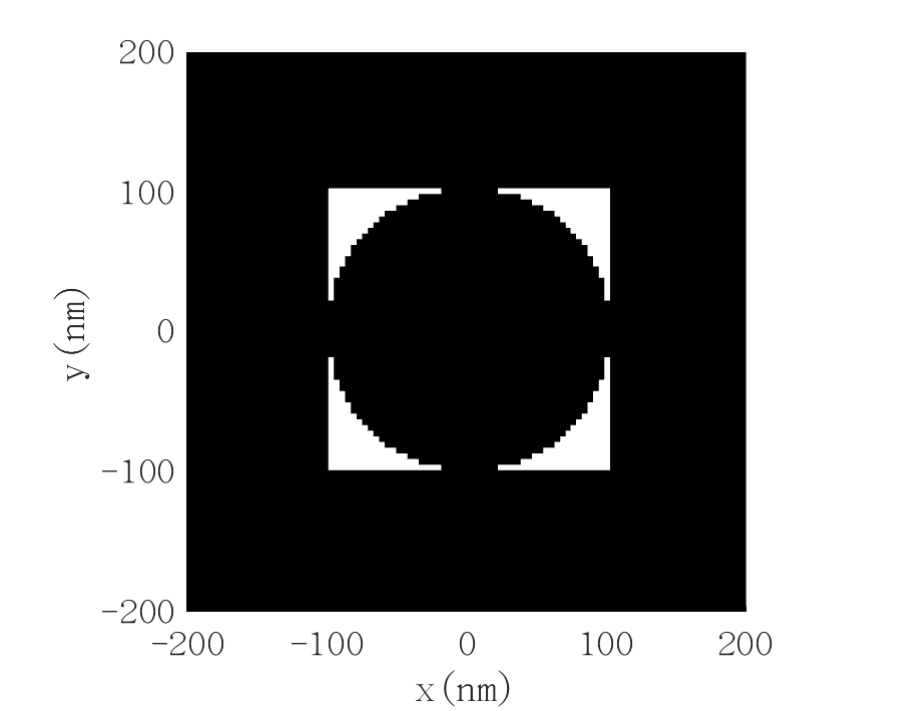}
			\caption{initial EPE}
			\label{fig1_5}
		\end{subfigure}
		
		\begin{subfigure}[b]{0.32\textwidth}
			\centering
			\includegraphics[width=\textwidth]{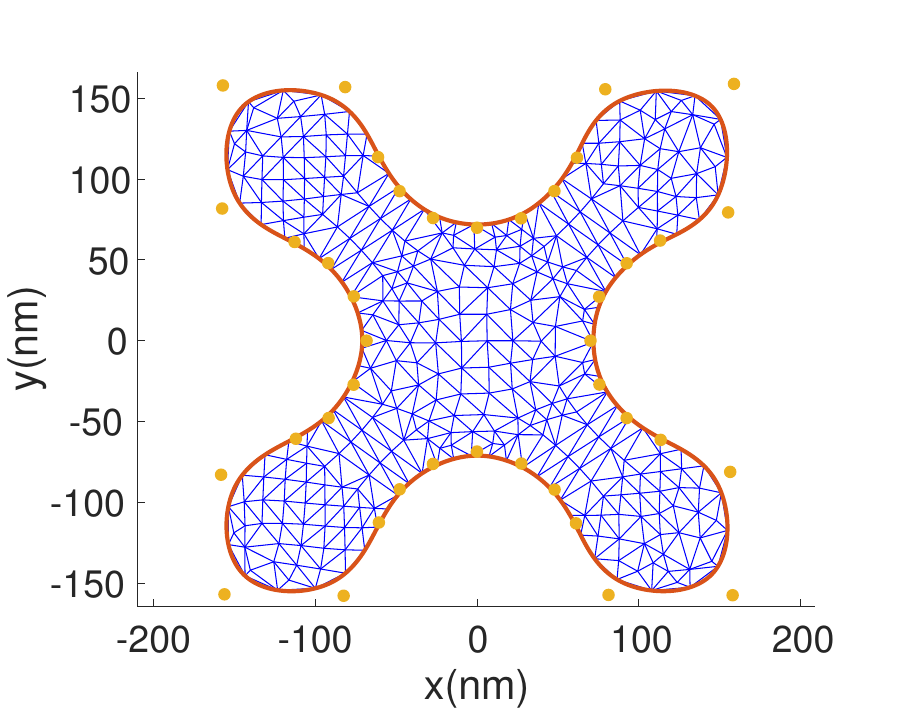}
			\caption{optimized mask}
			\label{fig1_6}
		\end{subfigure}
		\begin{subfigure}[b]{0.32\textwidth}
			\centering
			\includegraphics[width=\textwidth]{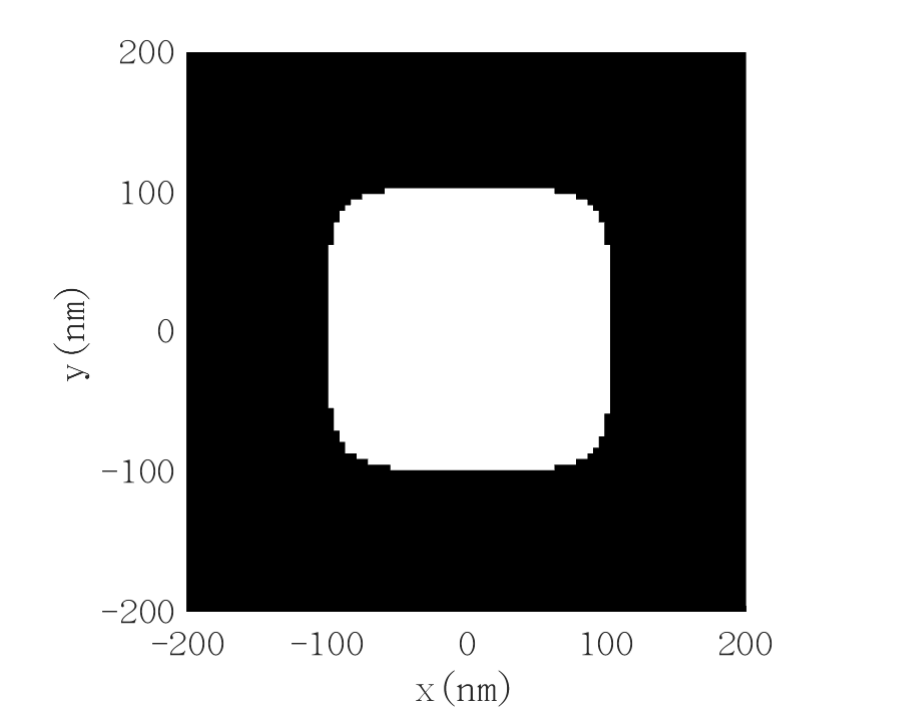}
			\caption{output of optimized mask}
			\label{fig1_7}
		\end{subfigure}
		\begin{subfigure}[b]{0.32\textwidth}
			\centering
			\includegraphics[width=\textwidth]{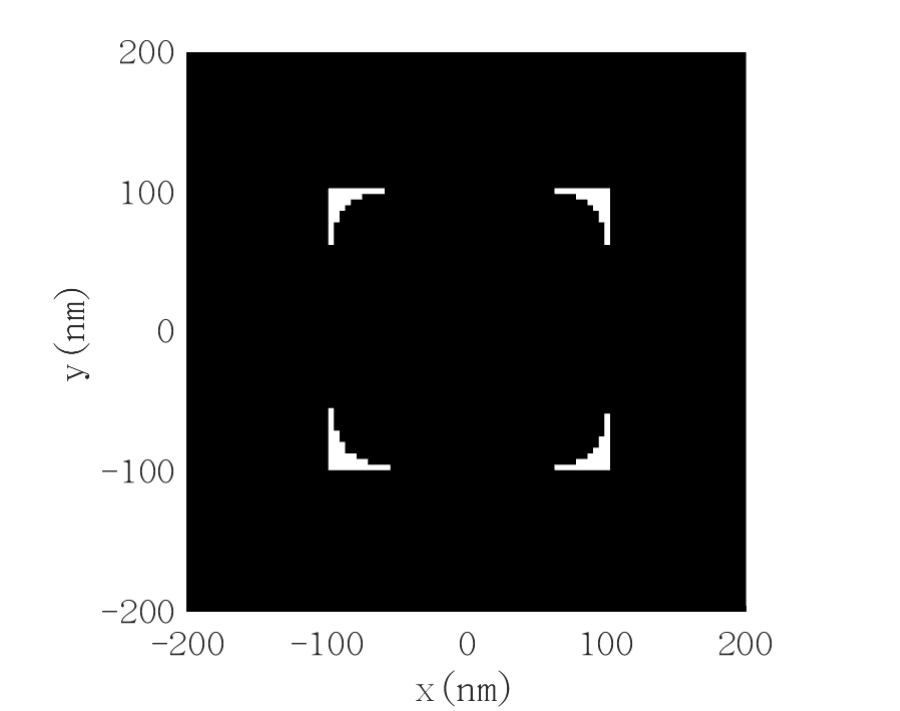}
			\caption{optimized EPE}
			\label{fig1_8}
		\end{subfigure}
		\caption{Example of square pattern} 
		\label{square}
	\end{figure}
    
\paragraph{Example 1.} In Figure \ref{square}, we use the simplest square target pattern as our first example. In this example, we arrange 40 control points on the boundary, 100 discrete points on the spline curve, and set the image size by $100 \times 100$ pixels. The resolution in this example is 4nm/pixel. It took 5144 seconds to finish the 100-step optimization process. In Figure (\ref{fig1_1}) we show the target square pattern. In Figure (\ref{fig1_2}) we show how the objective function decreases as the iteration goes on. We can see from the error decreasing graph that an almost optimal mask can be obtained in about thirty steps, so in fact, we only need to spend about 1500 seconds to get the optimized mask. Figure (\ref{fig1_3}) shows the initial mask pattern, in which we indicate the control points, curvilinear mask boundary, and the triangulation of the initial mask. The initial locations of control points are determined by the boundary of target pattern. Figure (\ref{fig1_4}) shows the output patterns with initial mask pattern as input. It is almost a circle, which differs a lot from the target. Figure (\ref{fig1_5}) shows the edge placement error map of initial mask. 
Figures (\ref{fig1_6}-\ref{fig1_8}) are corresponding pictures for the optimized mask. Among them, 
Figure (\ref{fig1_6}) shows the optimized mask with control points, curvilinear mask boundary and triangulation. It is quite different from the initial one in Figure (\ref{fig1_3}). 
From Figure (\ref{fig1_7}) and (\ref{fig1_8}), we can find that the image pattern is better than that in Figure (\ref{fig1_4}) and the edge placement error is reduced significantly, which confirms the feasibility of our method.

\begin{figure}[htbp]
	\centering
	
	\begin{subfigure}[b]{0.32\textwidth}
		\centering
		\includegraphics[width=\textwidth]{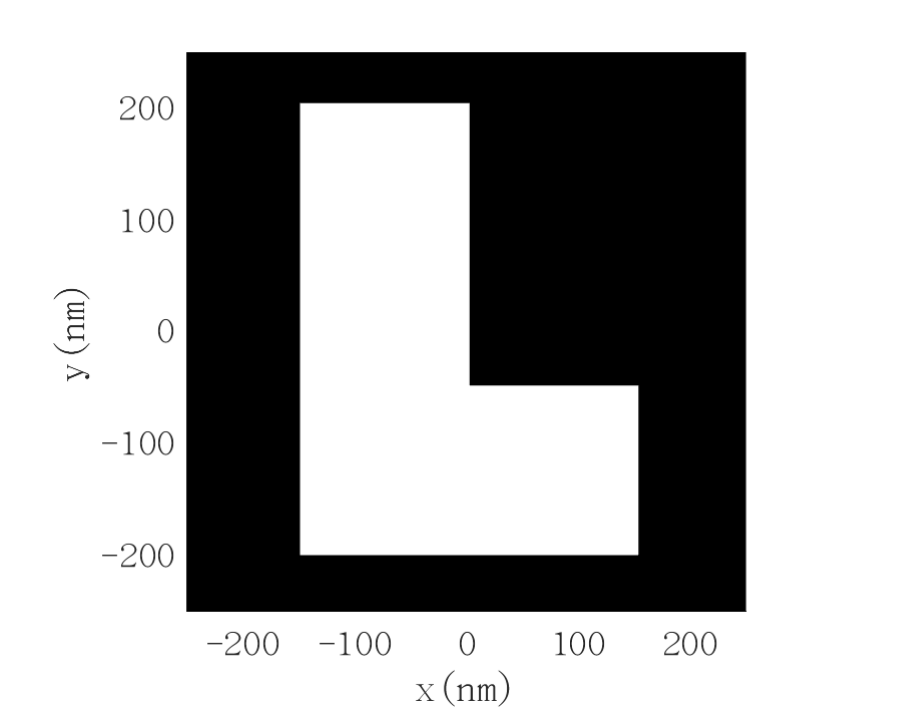}
		\caption{target pattern}
		\label{fig2_1}
	\end{subfigure}
	\begin{subfigure}[b]{0.32\textwidth}
		\centering
		\includegraphics[width=\textwidth]{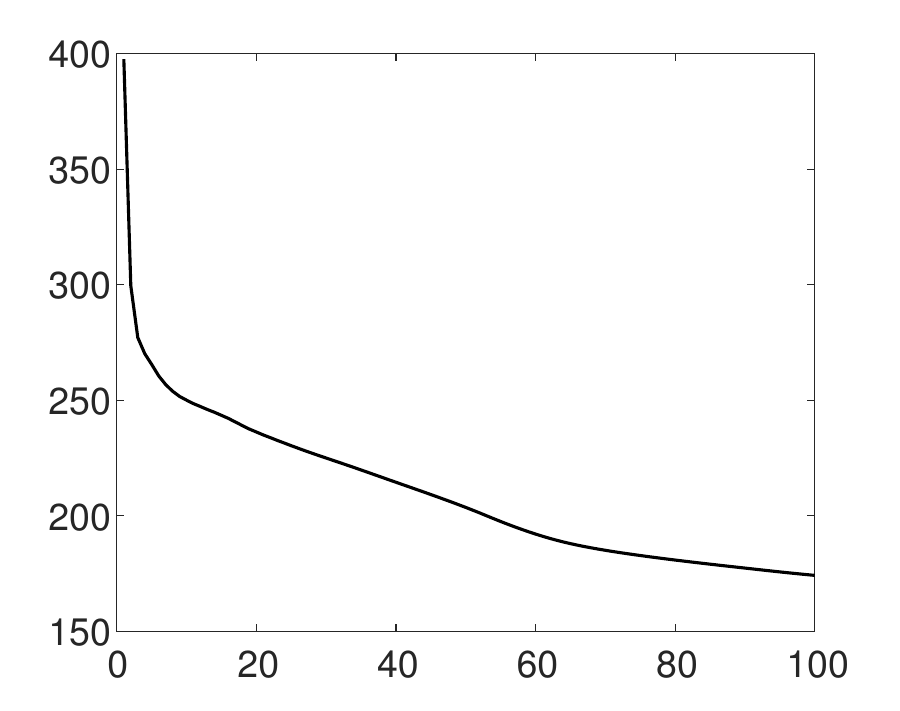}
		\caption{error decreasing}
		\label{fig2_2}
	\end{subfigure}
	
	\begin{subfigure}[b]{0.32\textwidth}
		\centering
		\includegraphics[width=\textwidth]{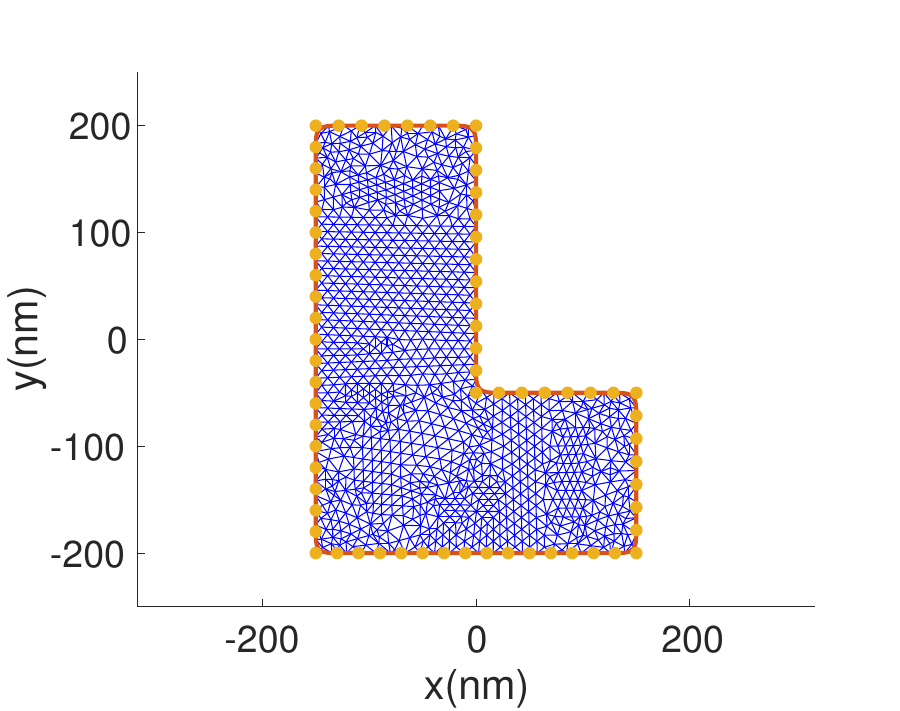}
		\caption{initial mask}
		\label{fig2_3}
	\end{subfigure}
	\begin{subfigure}[b]{0.32\textwidth}
		\centering
		\includegraphics[width=\textwidth]{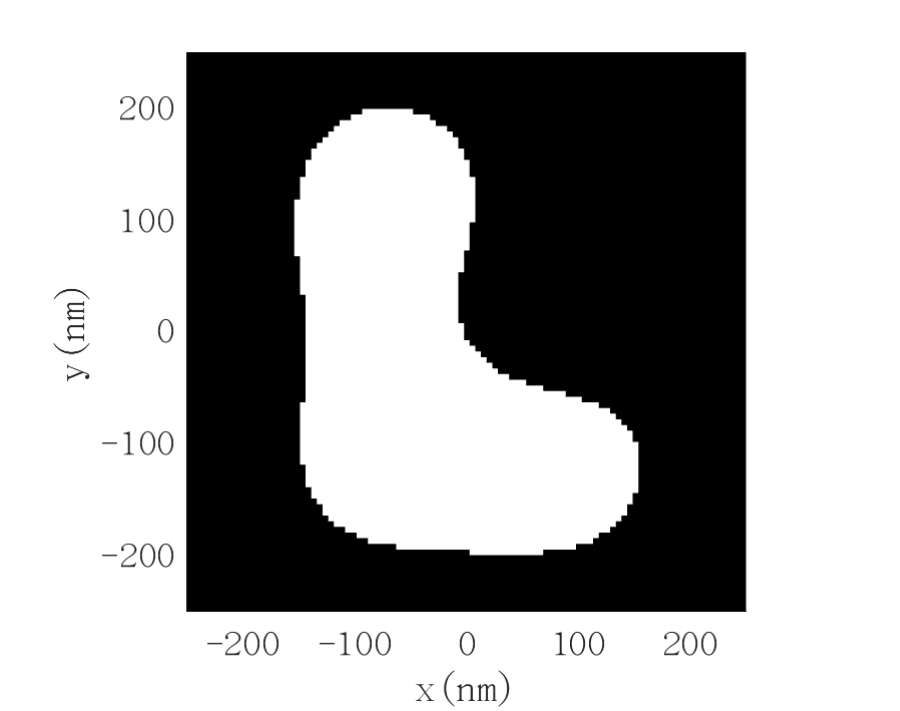}
		\caption{output of initial mask}
		\label{fig2_4}
	\end{subfigure}
	\begin{subfigure}[b]{0.32\textwidth}
		\centering
		\includegraphics[width=\textwidth]{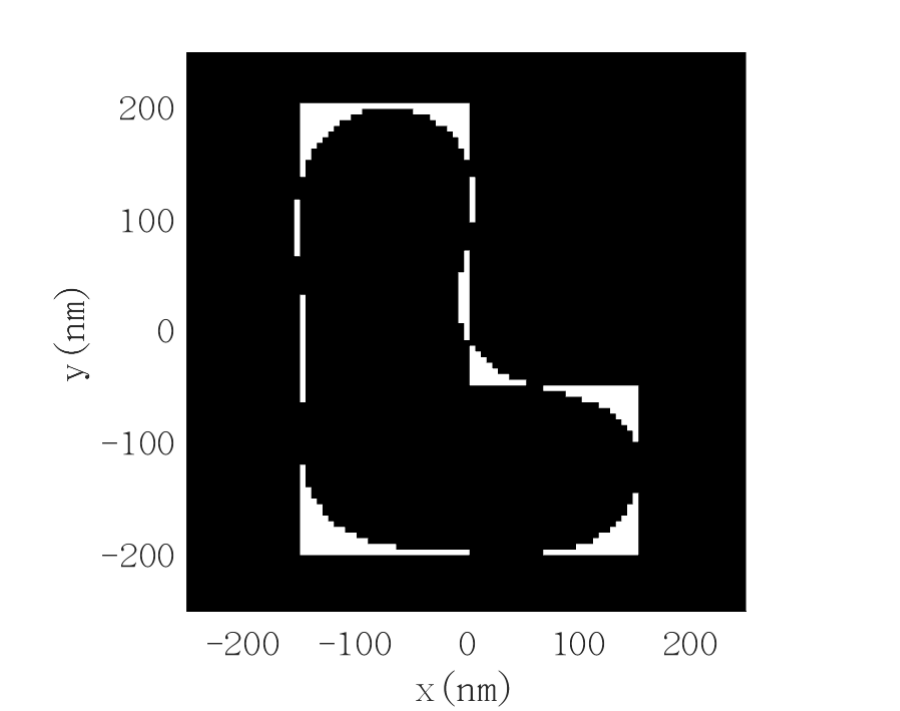}
		\caption{initial EPE}
		\label{fig2_5}
	\end{subfigure}
	
	\begin{subfigure}[b]{0.32\textwidth}
		\centering
		\includegraphics[width=\textwidth]{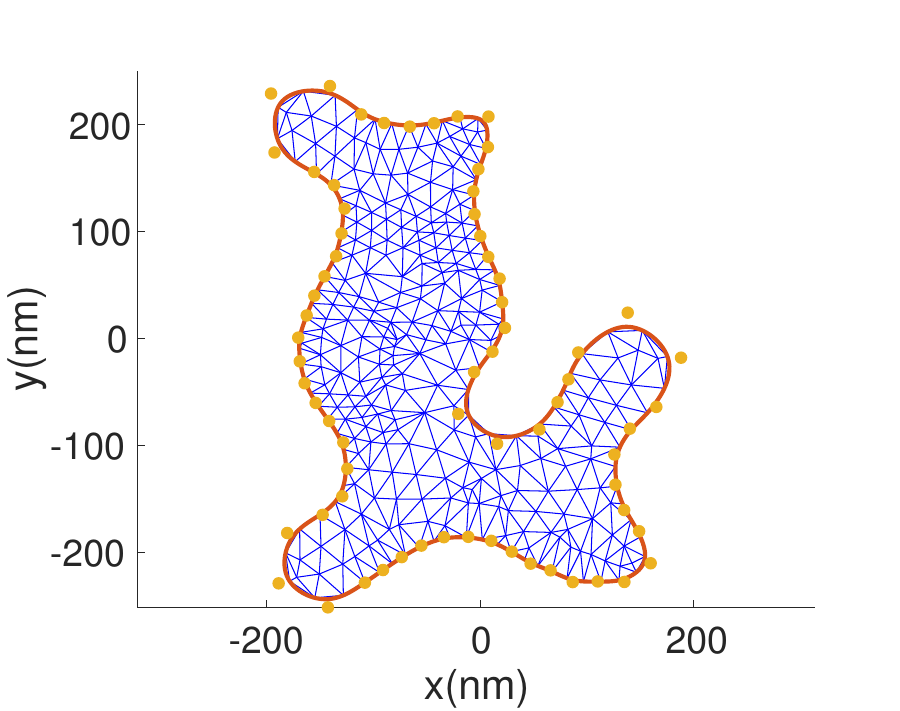}
		\caption{optimized mask}
		\label{fig2_6}
	\end{subfigure}
	\begin{subfigure}[b]{0.32\textwidth}
		\centering
		\includegraphics[width=\textwidth]{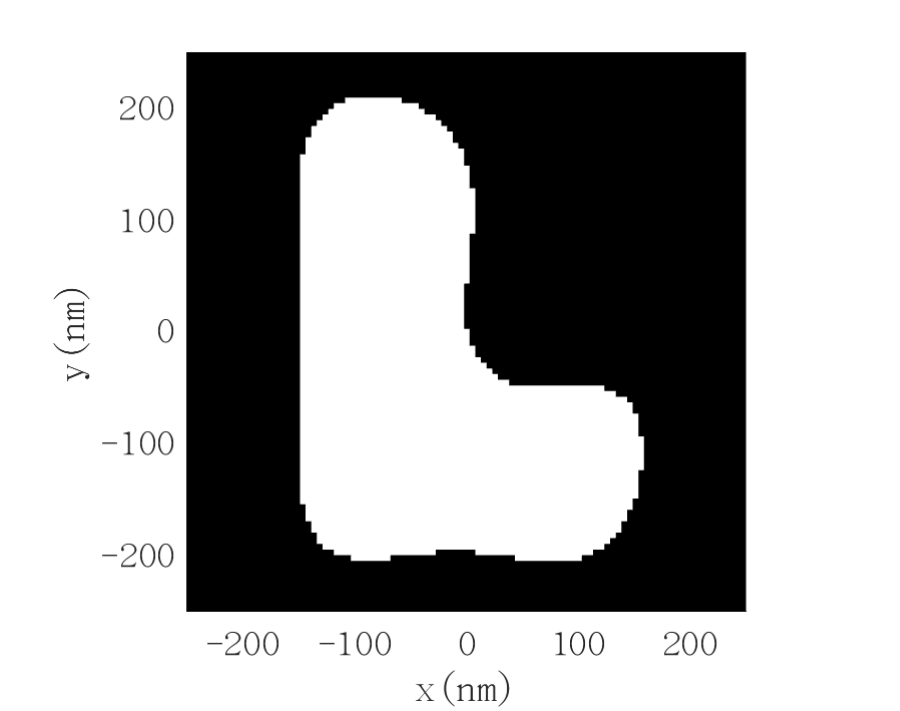}
		\caption{output of optimized mask}
		\label{fig2_7}
	\end{subfigure}
	\begin{subfigure}[b]{0.32\textwidth}
		\centering
		\includegraphics[width=\textwidth]{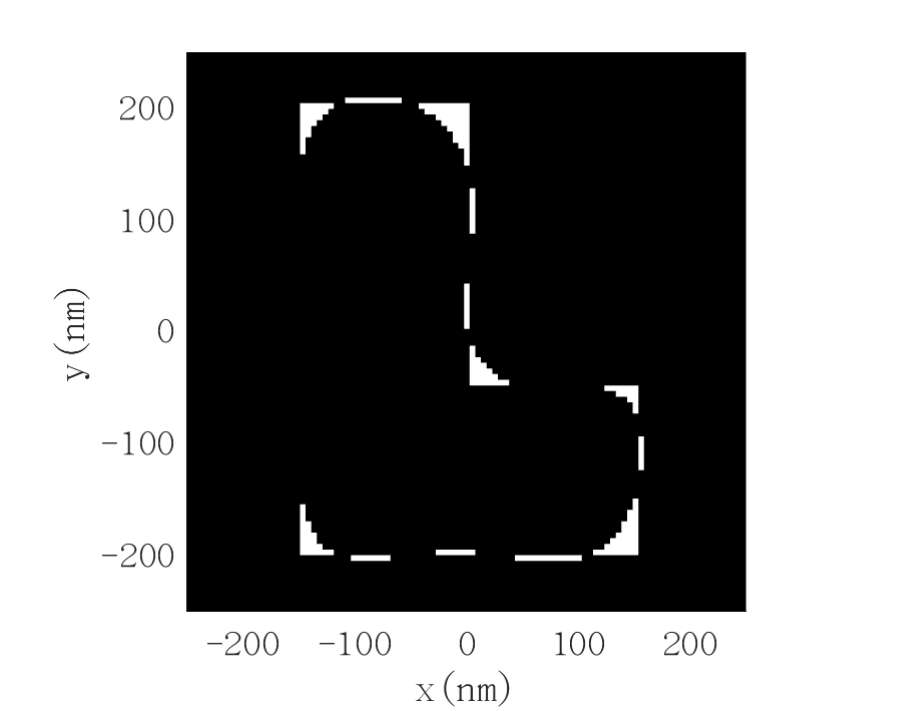}
		\caption{optimized EPE}
		\label{fig2_8}
	\end{subfigure}
	\caption{Example of L-shape pattern} 
	\label{Lshape}
\end{figure}

The arrangements of the following Figures \ref{Lshape},~\ref{twocolumn},~\ref{reversal pattern},~\ref{bigmac} are similar. For each target pattern, we present the initial mask along with its output, the optimized mask and its corresponding output. To more comprehensively assess the imaging quality, we also supply the edge placement error patterns for both the initial and optimized masks. Moreover, we provide a figure to illustrate the reduction of the error during the optimization process.

\paragraph{Example 2}In Figure \ref{Lshape} we use a less symmetrical L-shaped pattern as an example to verify the adaptability of the model to patterns with inside and outside corners. We arrange 68 control points on the boundary, 100 discrete points on the spline curve, and set the image size by $100 \times 100$ pixels. The resolution in this example is 5nm/pixel. It took 7845 seconds to finish the 100-step optimization process. In the initial output, the output pattern shrinks significantly at each end and the inner corner is smoothed. Our method improves the imaging performance at each corner and makes the image more similar to the target pattern, which verified the adaptability to the asymmetrical pattern of our method.

\begin{figure}[htbp]
	\centering

	\begin{subfigure}[b]{0.32\textwidth}
		\centering
		\includegraphics[width=\textwidth]{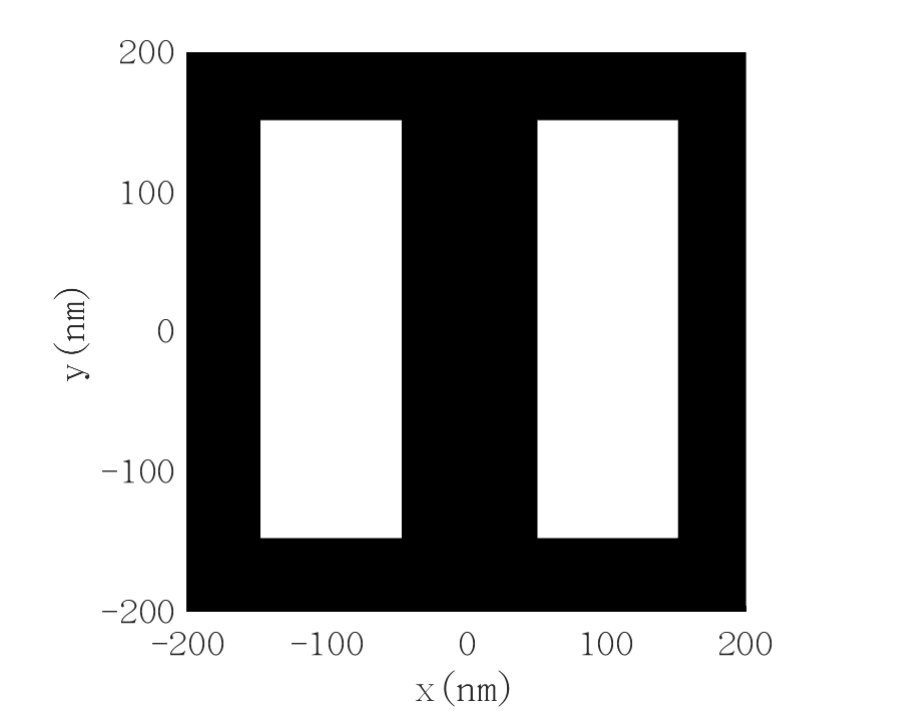}
		\caption{target pattern}
		\label{fig3_1}
	\end{subfigure}
	\begin{subfigure}[b]{0.32\textwidth}
		\centering
		\includegraphics[width=\textwidth]{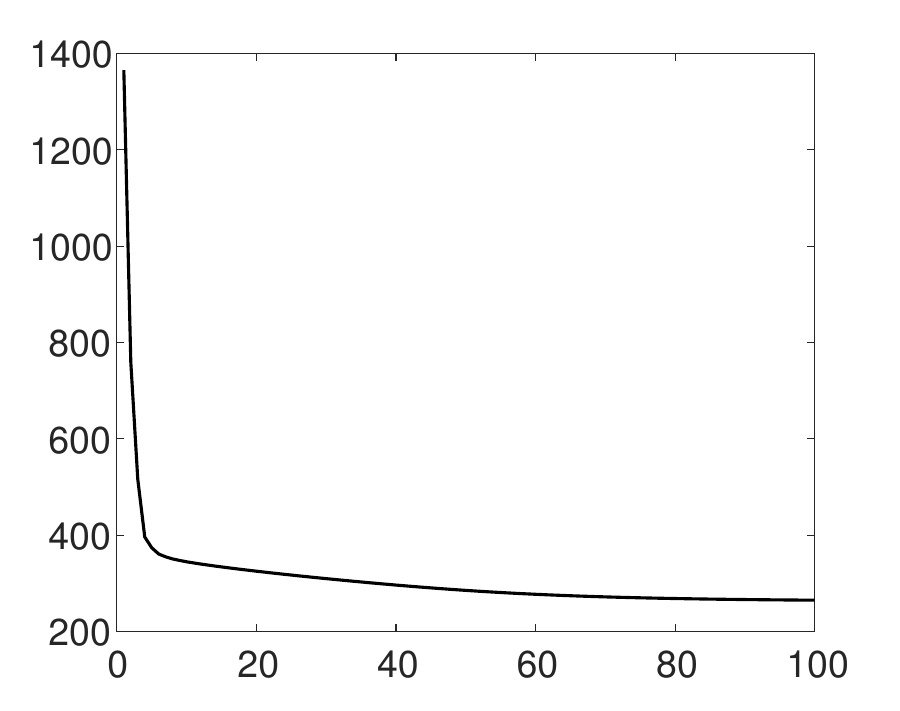}
		\caption{error decreasing}
		\label{fig3_2}
	\end{subfigure}
	
	\begin{subfigure}[b]{0.32\textwidth}
		\centering
		\includegraphics[width=\textwidth]{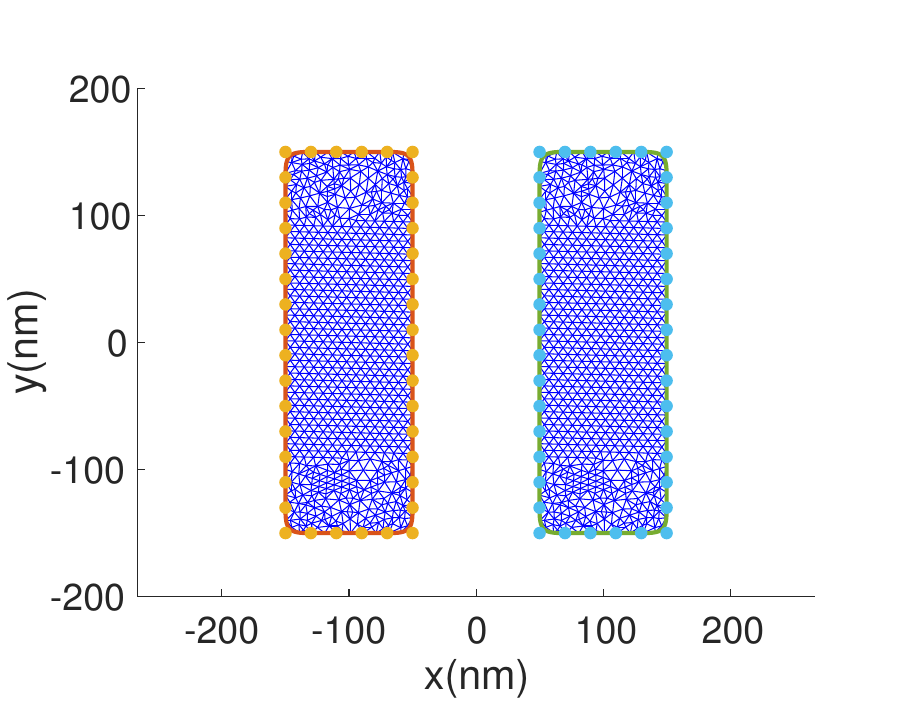}
		\caption{initial mask}
		\label{fig3_3}
	\end{subfigure}
	\begin{subfigure}[b]{0.32\textwidth}
		\centering
		\includegraphics[width=\textwidth]{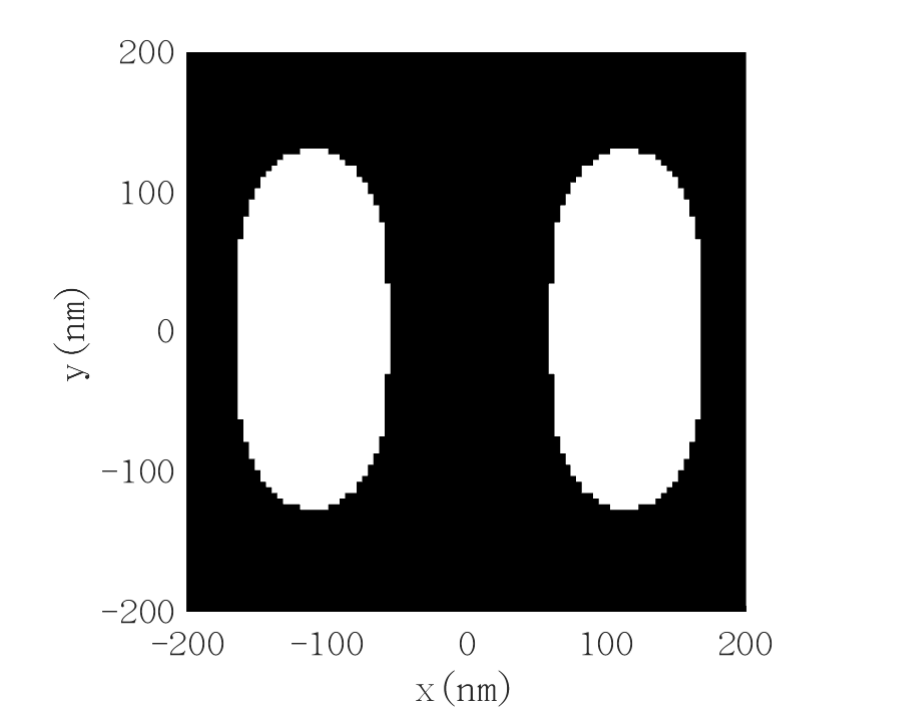}
		\caption{output of initial mask}
		\label{fig3_4}
	\end{subfigure}
	\begin{subfigure}[b]{0.32\textwidth}
		\centering
		\includegraphics[width=\textwidth]{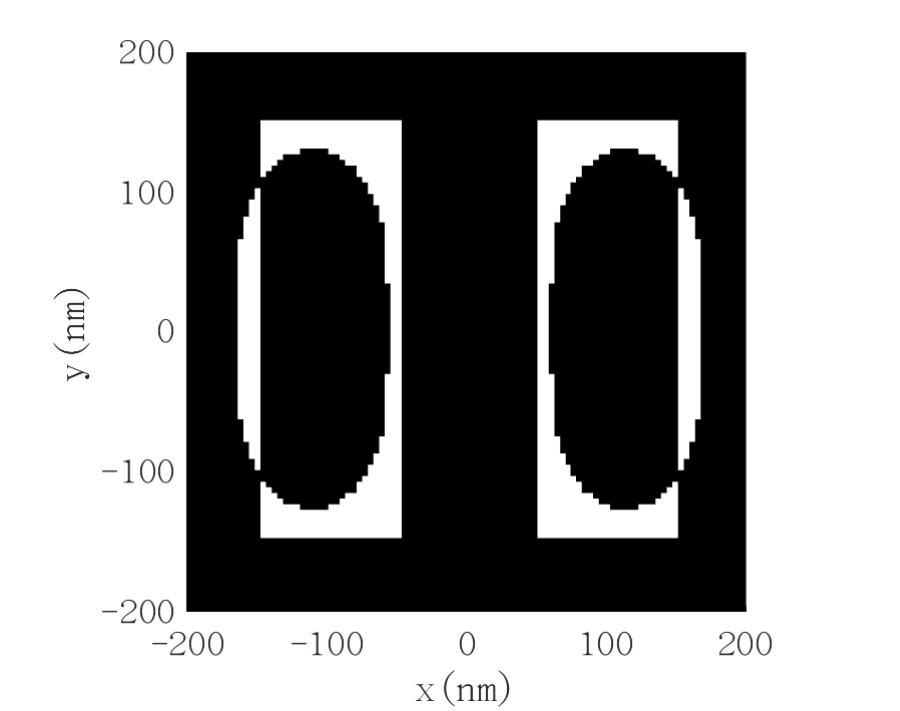}
		\caption{initial EPE}
		\label{fig3_5}
	\end{subfigure}
	
	\begin{subfigure}[b]{0.32\textwidth}
		\centering
		\includegraphics[width=\textwidth]{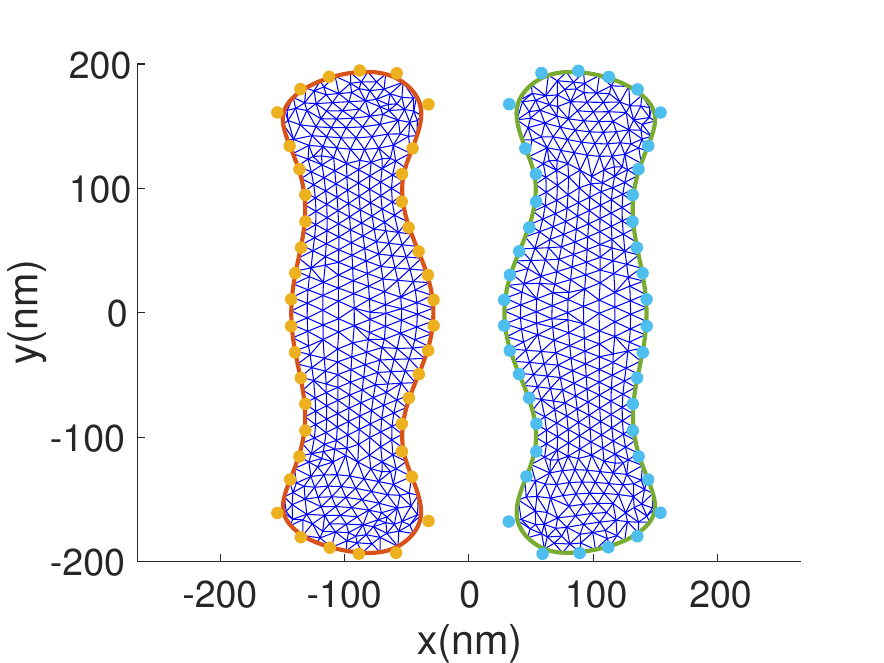}
		\caption{optimized mask}
		\label{fig3_6}
	\end{subfigure}
	\begin{subfigure}[b]{0.32\textwidth}
		\centering
		\includegraphics[width=\textwidth]{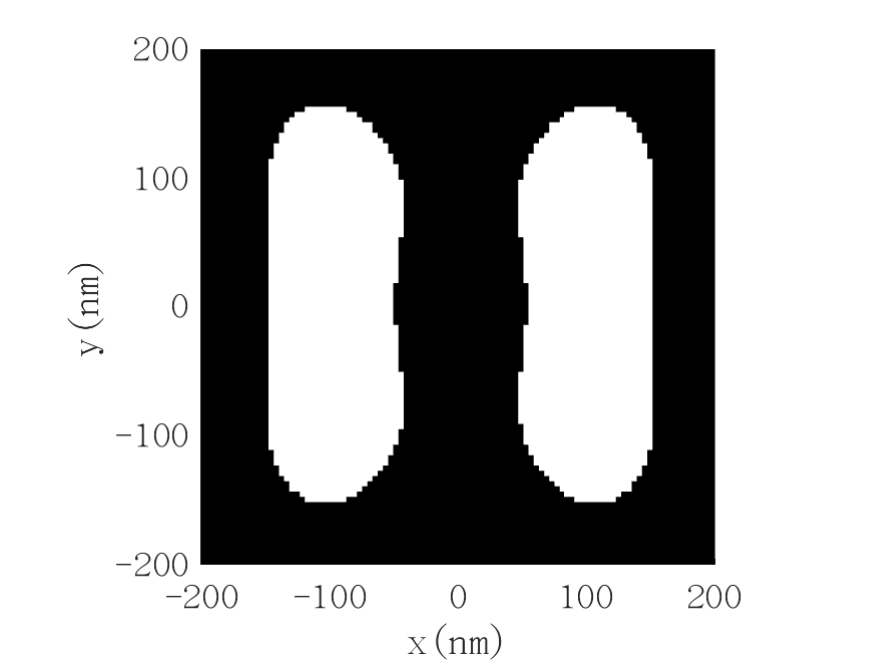}
		\caption{output of optimized mask}
		\label{fig3_7}
	\end{subfigure}
	\begin{subfigure}[b]{0.32\textwidth}
		\centering
		\includegraphics[width=\textwidth]{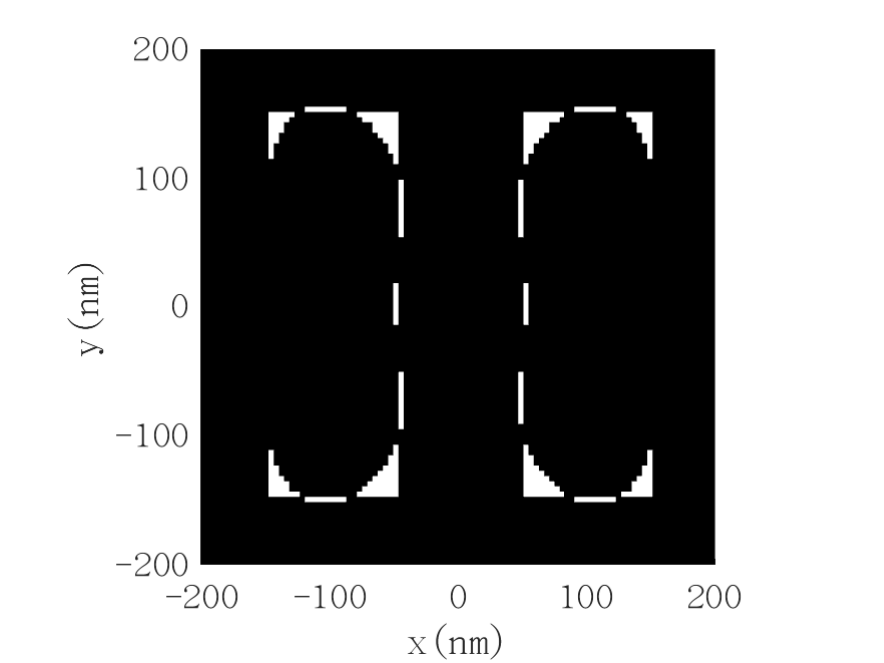}
		\caption{optimized EPE}
		\label{fig3_8}
	\end{subfigure}
	\caption{Example of two rectangle pattern} 
	\label{twocolumn}
\end{figure}

\paragraph{Example 3}In Figure \ref{twocolumn}, we use a pattern of two separate rectangles as an example. Here we arrange 40 control points for each region on the boundary, 80 discrete points on each spline curve, and set the image size by $100 \times 100$ pixels. The resolution in this example is 4nm/pixel. It took 14040 seconds to finish the 100-step optimization process. At each end, the output pattern of the initial mask shrinks significantly and the position of output is away from the original design. This means that the initial mask is not suitable for generating the desired target pattern. After optimization by our method, both the shortening and offset of the initial mask are compensated. This example means that our method can be successfully used for multiple regions.

\begin{figure}[htbp]
	\centering
    
	\begin{subfigure}[b]{0.32\textwidth}
		\centering
		\includegraphics[width=\textwidth]{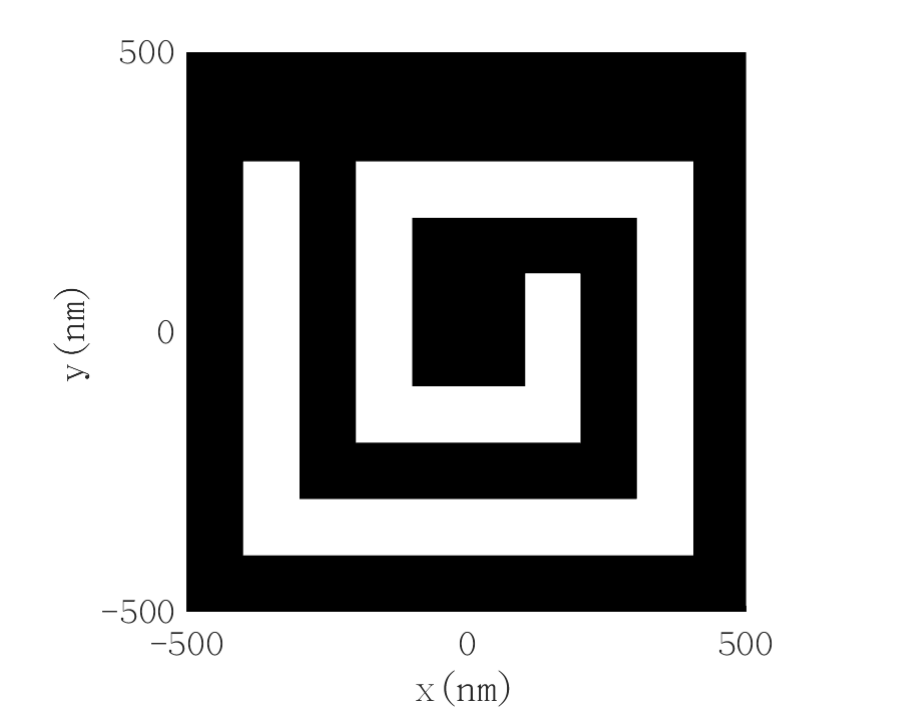}
		\caption{target pattern}
		\label{fig4_1}
	\end{subfigure}
	\begin{subfigure}[b]{0.32\textwidth}
		\centering
		\includegraphics[width=\textwidth]{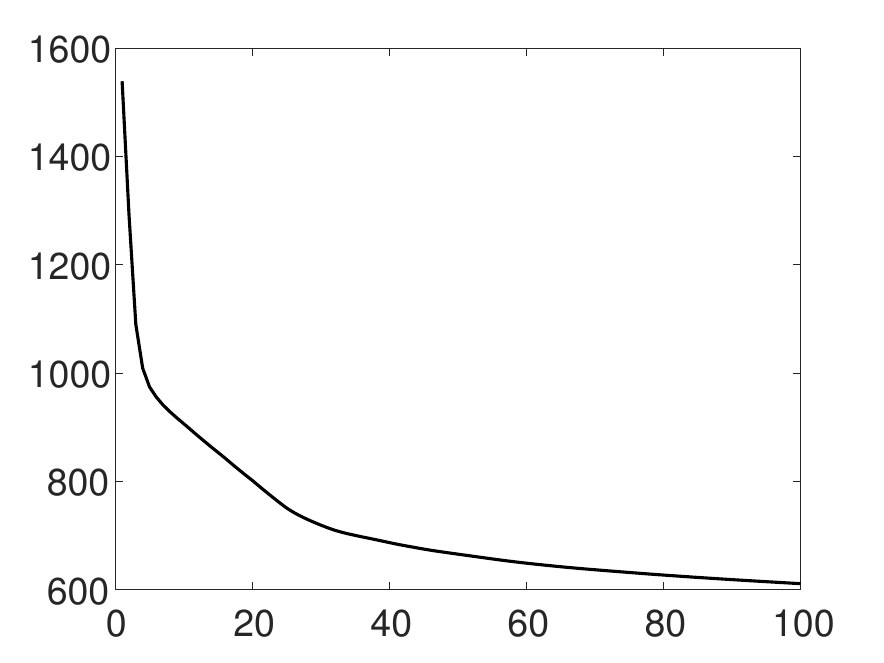}
		\caption{error decreasing}
		\label{fig4_2}
	\end{subfigure}
	
	\begin{subfigure}[b]{0.32\textwidth}
		\centering
		\includegraphics[width=\textwidth]{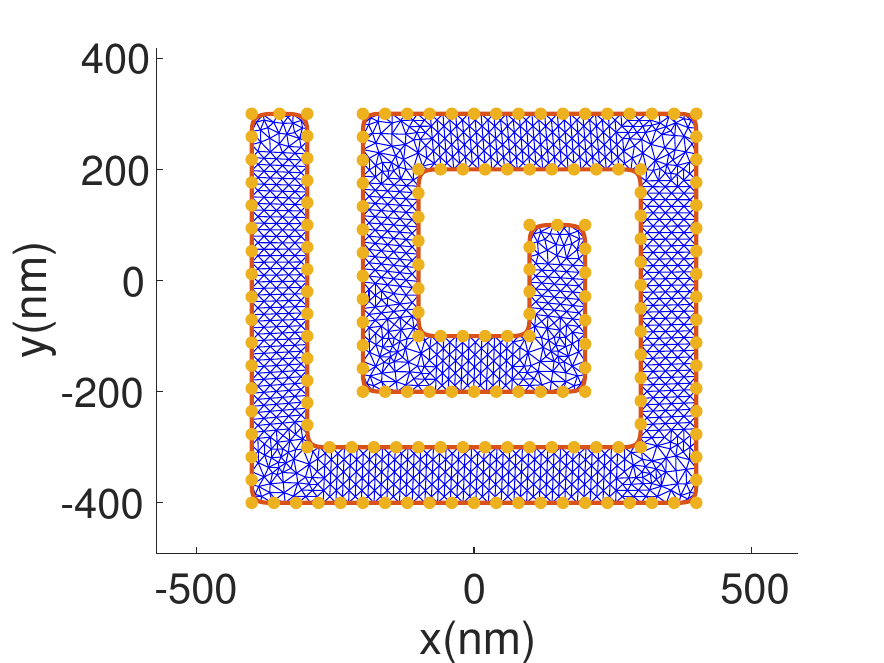}
		\caption{initial mask}
		\label{fig4_3}
	\end{subfigure}
	\begin{subfigure}[b]{0.32\textwidth}
		\centering
		\includegraphics[width=\textwidth]{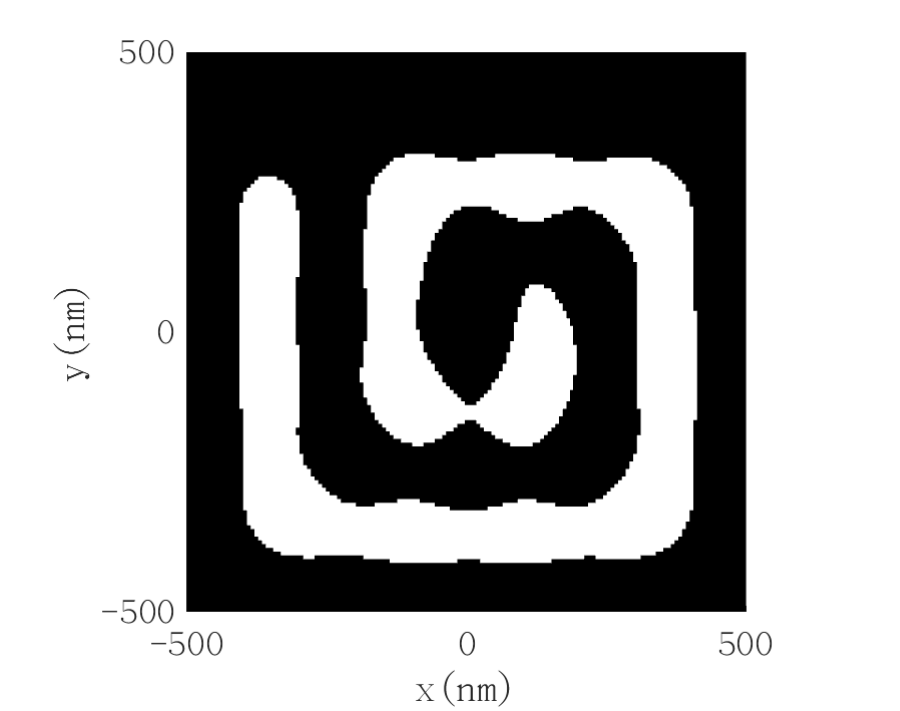}
		\caption{output of initial mask}
		\label{fig4_4}
	\end{subfigure}
	\begin{subfigure}[b]{0.32\textwidth}
		\centering
		\includegraphics[width=\textwidth]{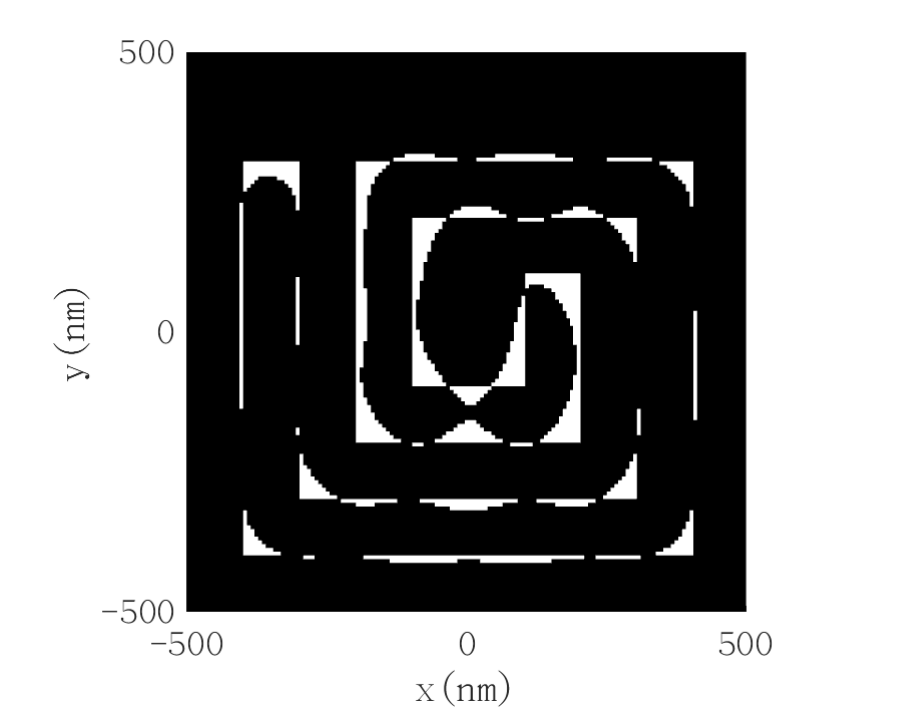}
		\caption{initial EPE}
		\label{fig4_5}
	\end{subfigure}
	
	\begin{subfigure}[b]{0.32\textwidth}
		\centering
		\includegraphics[width=\textwidth]{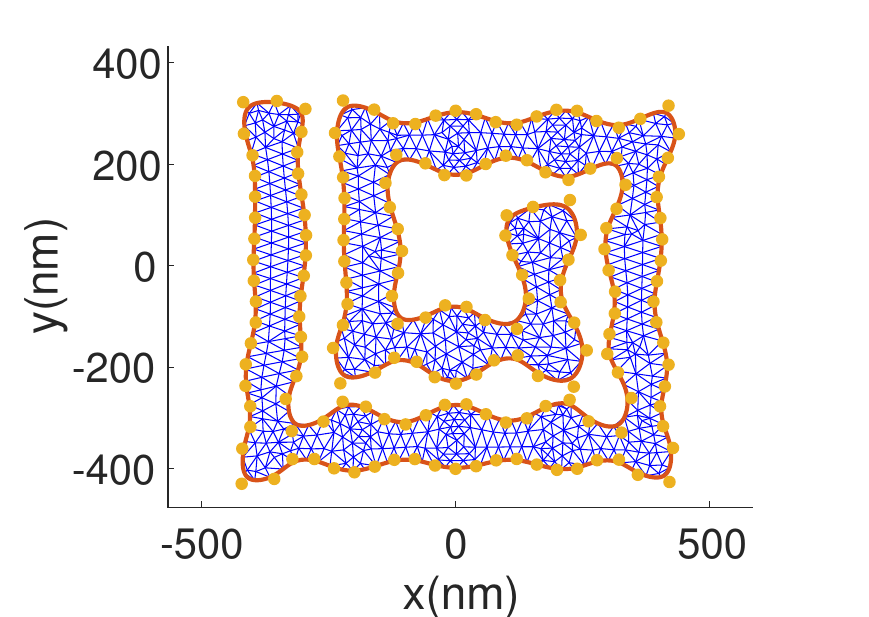}
		\caption{optimized mask}
		\label{fig4_6}
	\end{subfigure}
	\begin{subfigure}[b]{0.32\textwidth}
		\centering
		\includegraphics[width=\textwidth]{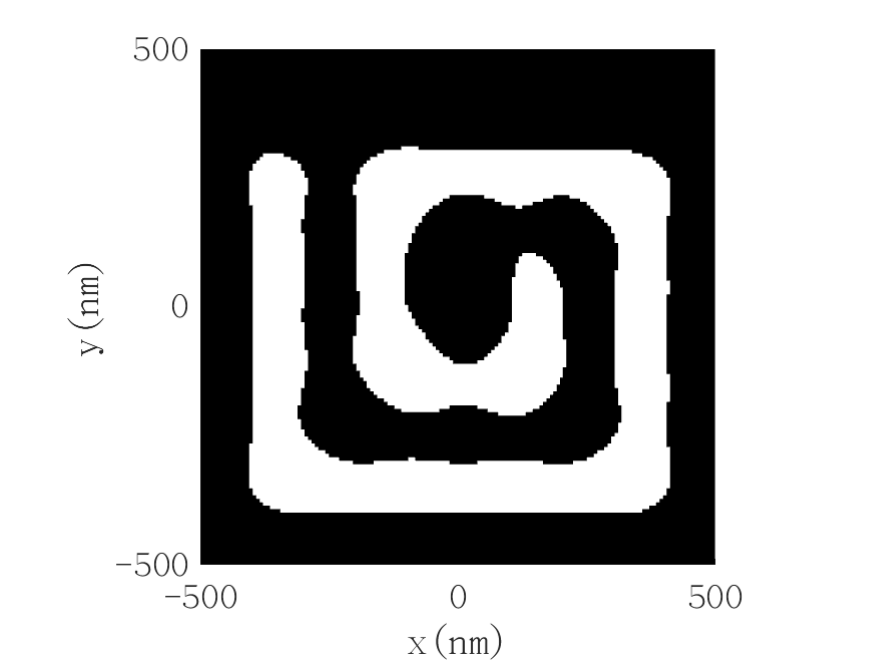}
		\caption{output of optimized mask}
		\label{fig4_7}
	\end{subfigure}
	\begin{subfigure}[b]{0.32\textwidth}
		\centering
		\includegraphics[width=\textwidth]{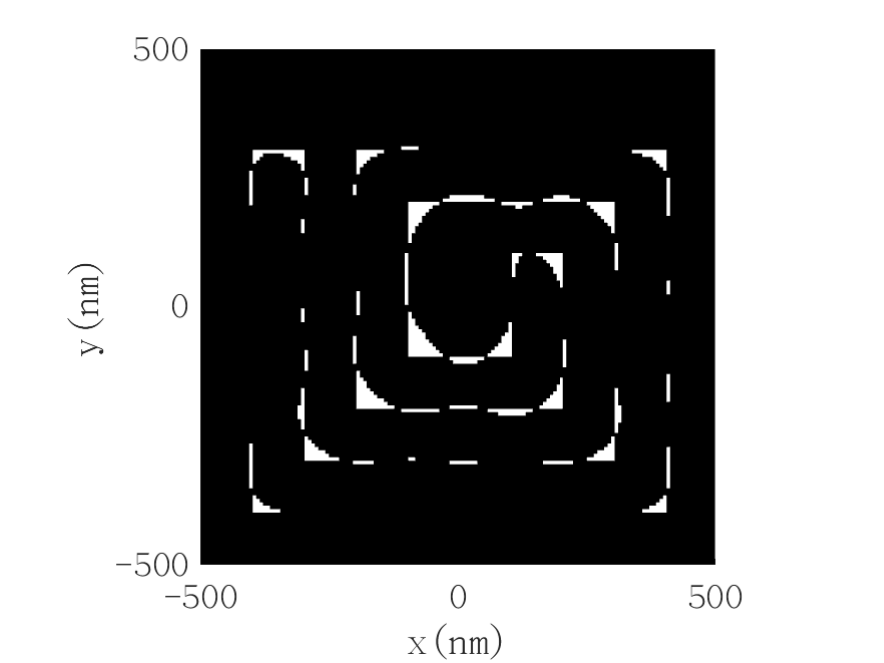}
		\caption{optimized EPE}
		\label{fig4_8}
	\end{subfigure}
	\caption{Example of reversal pattern} 
	\label{reversal pattern}
\end{figure}

\paragraph{Example 4}In Figure \ref{reversal pattern}, we use a relatively complex reversal pattern as an example. In this example, we arrange 171 control points on the boundary, 300 discrete points on the spline curve, and set the image size by $150 \times 150$ pixels. The resolution in this example is about 6.67nm/pixel. It took 38583 seconds to complete the 100-step optimization process. From the initial EPE in Figure (\ref{fig4_5}) we can see some obvious shortcomings of the output. At the top left of the Figure (\ref{fig4_5}), we can see the shortening of the pattern. At the center of the same picture, we can see a severe narrowing, which is unacceptable in optical lithography. Here is still some offset at the center part. With the improvement of our method, the shortcomings mentioned above are noticeably remedied. The shrinkage at each end of the line is compensated. The severe narrowing of line width at the center of the pattern is corrected greatly, and the offset of the pattern is also corrected in the final result.

\begin{figure}[htbp]
	\centering
	
	\begin{subfigure}[b]{0.32\textwidth}
		\centering
		\includegraphics[width=\textwidth]{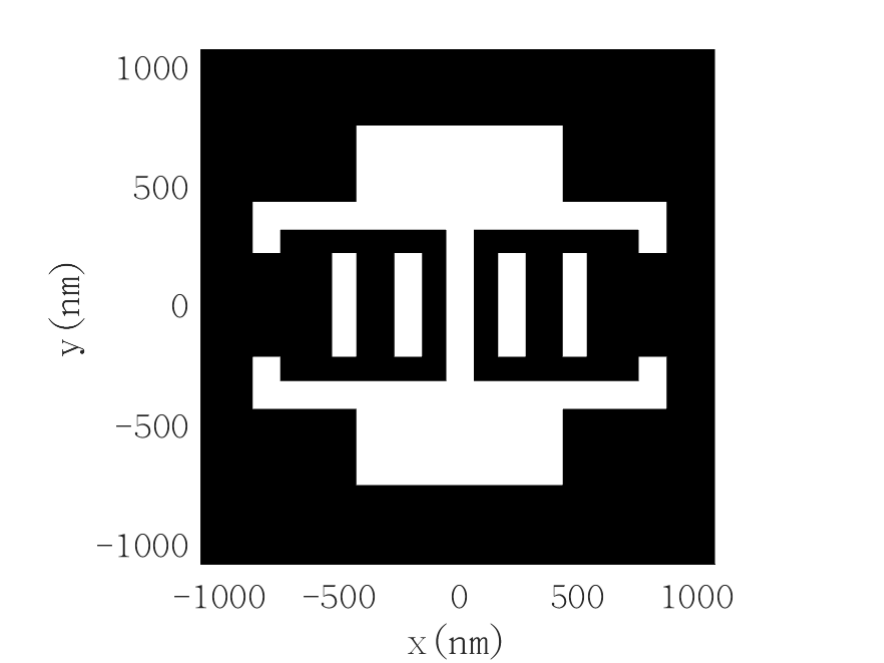}
		\caption{target pattern}
		\label{fig5_1}
	\end{subfigure}
	\begin{subfigure}[b]{0.32\textwidth}
		\centering
		\includegraphics[width=\textwidth]{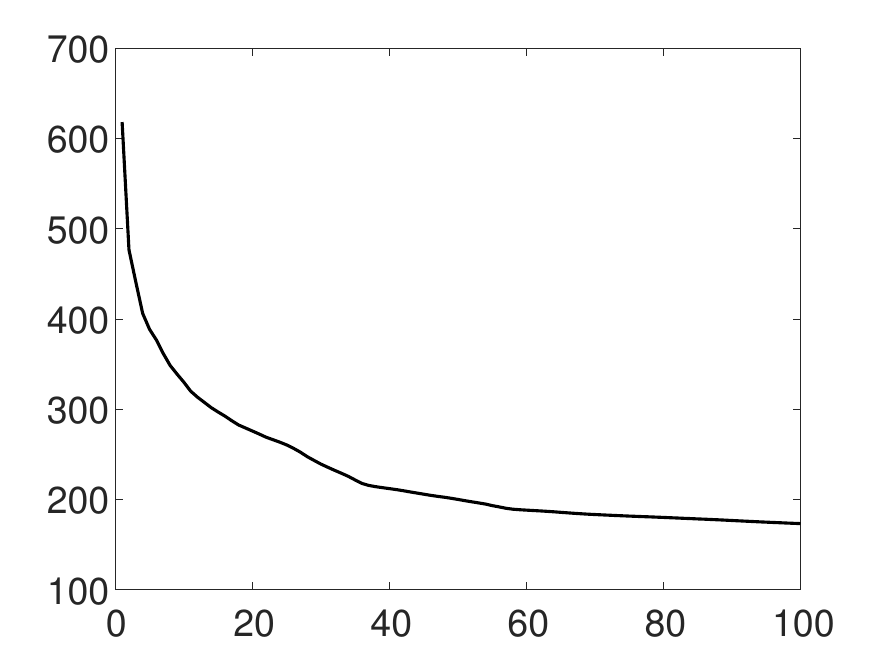}
		\caption{error decreasing}
		\label{fig5_2}
	\end{subfigure}
	
	\begin{subfigure}[b]{0.32\textwidth}
		\centering
		\includegraphics[width=\textwidth]{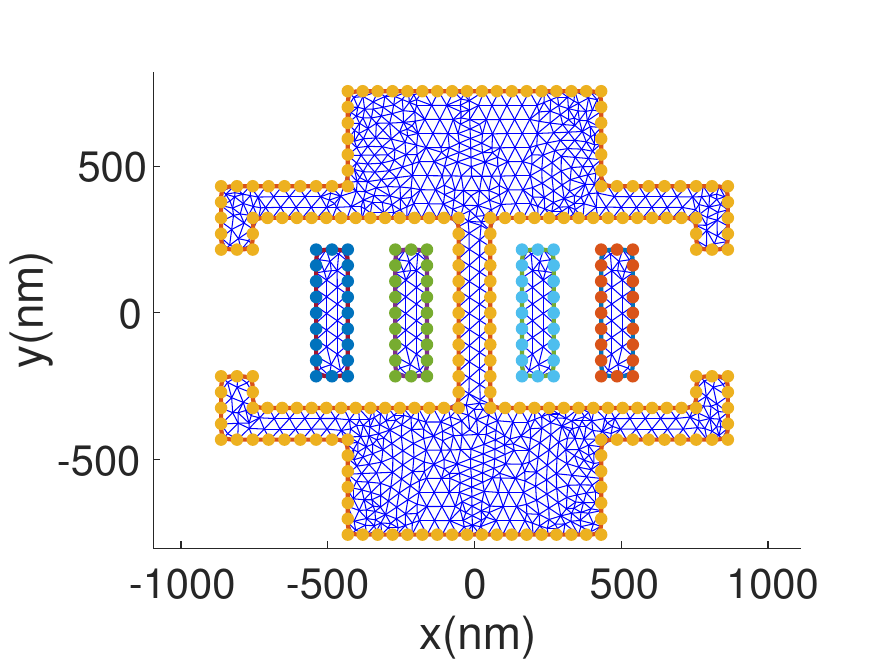}
		\caption{initial mask}
		\label{fig5_3}
	\end{subfigure}
	\begin{subfigure}[b]{0.32\textwidth}
		\centering
		\includegraphics[width=\textwidth]{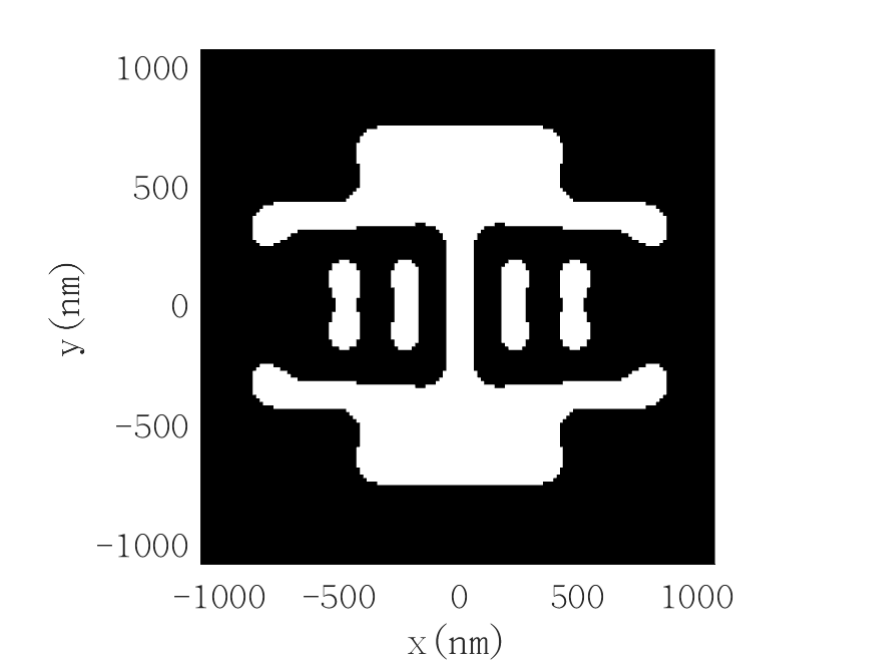}
		\caption{output of initial mask}
		\label{fig5_4}
	\end{subfigure}
	\begin{subfigure}[b]{0.32\textwidth}
		\centering
		\includegraphics[width=\textwidth]{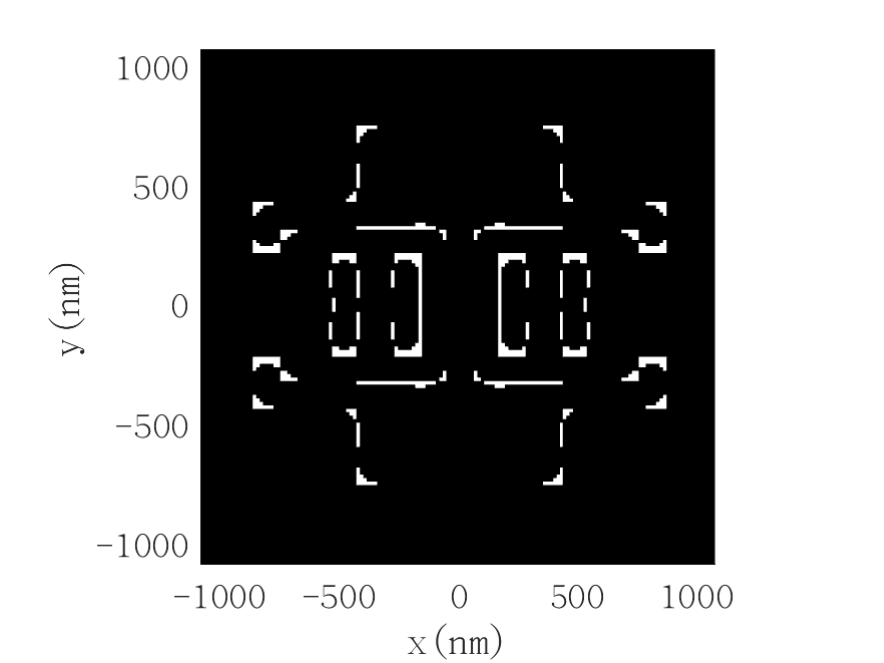}
		\caption{initial EPE}
		\label{fig5_5}
	\end{subfigure}
	
	\begin{subfigure}[b]{0.32\textwidth}
		\centering
		\includegraphics[width=\textwidth]{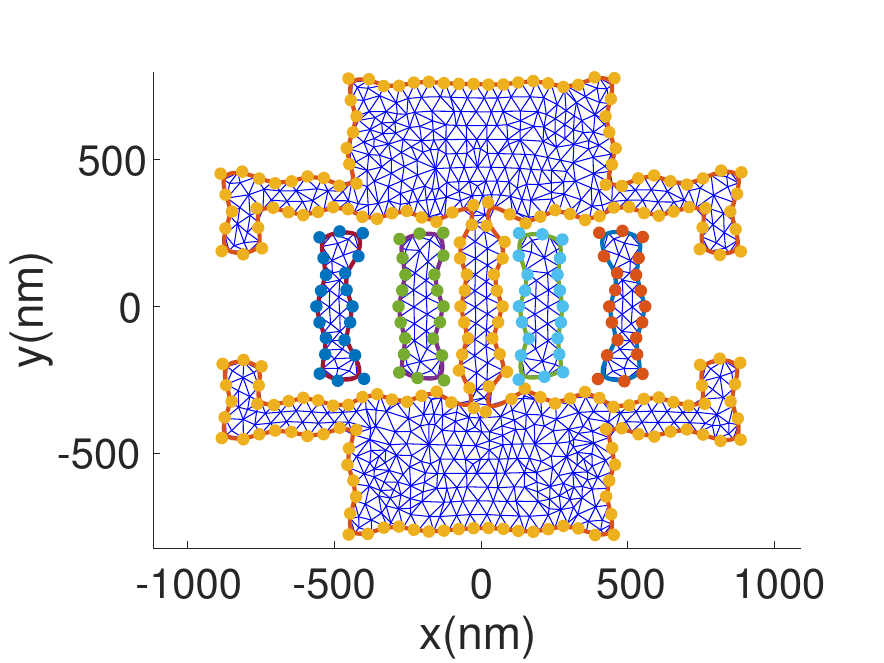}
		\caption{optimized mask}
		\label{fig5_6}
	\end{subfigure}
	\begin{subfigure}[b]{0.32\textwidth}
		\centering
		\includegraphics[width=\textwidth]{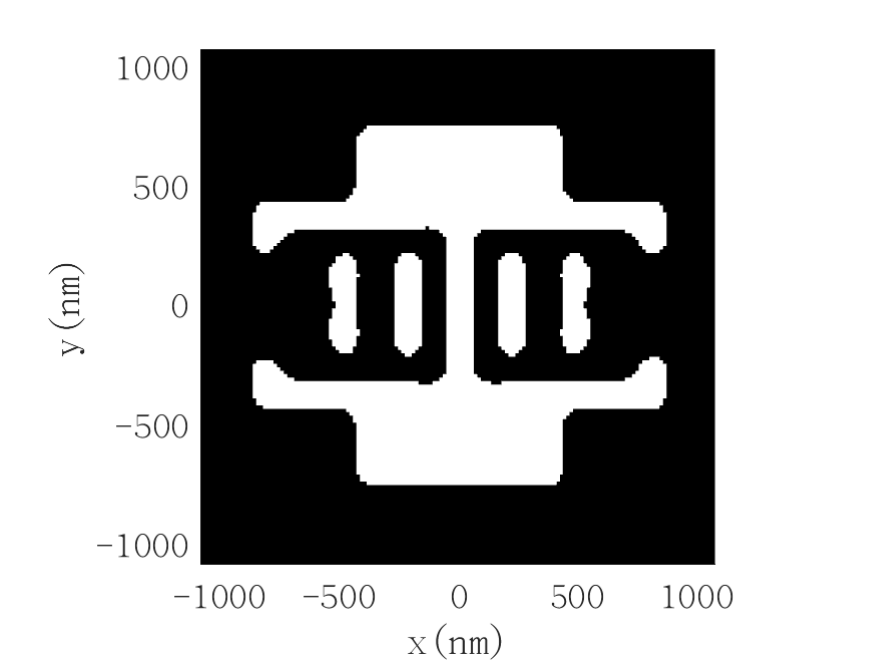}
		\caption{output of optimized mask}
		\label{fig5_7}
	\end{subfigure}
	\begin{subfigure}[b]{0.32\textwidth}
		\centering
		\includegraphics[width=\textwidth]{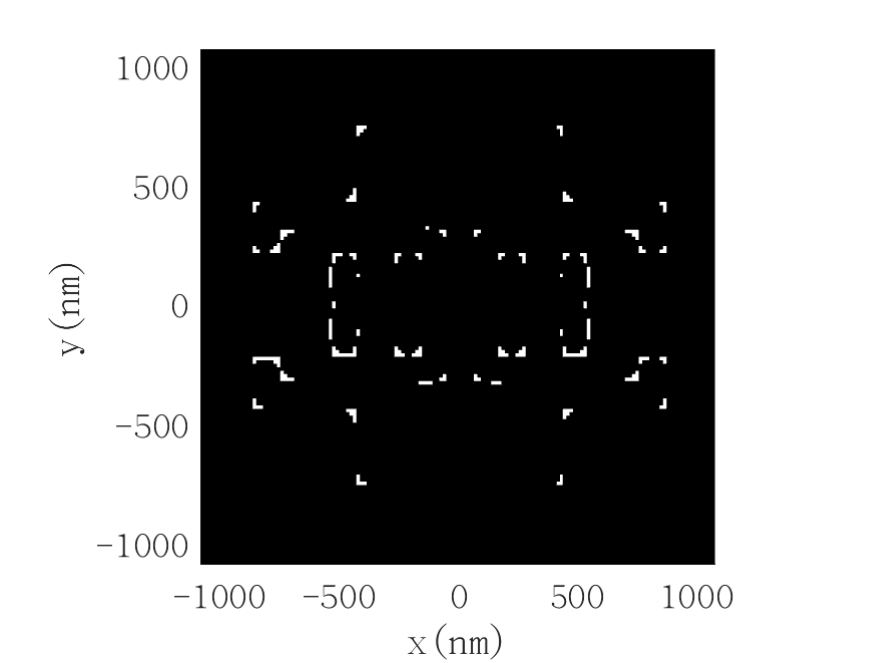}
		\caption{optimized EPE}
		\label{fig5_8}
	\end{subfigure}
	\caption{Example of BigMaC pattern} 
	\label{bigmac}
\end{figure}

\paragraph{Example 5}In the last example we use Big-Small-Feature \& Corner Test Mask (BigMaC)\cite{bigmac} pattern as an example. It is a mask with a large number of features coexisting simultaneously, including both large-scale and small-scale features, long-scale and short-scale features, inner and outer corners, as well as vertical and horizontal features. The detailed results are listed in Figure \ref{bigmac}. Here we extended the size of the pattern to 1.3 times. In this example the number of control points for the biggest region is 202, discrete points on this spline curve is 200. And for each small region, we arrange 20 control points on the boundary, 20 discrete points on the curve. We set the image size by $150 \times 150$ pixels. The resolution in this example is about 14.39nm/pixel. It took 41765 seconds to complete the 100-step optimization process. From the results in the last row of Figure \eqref{bigmac}, we can find that the optimization algorithm brings a satisfactory improvement on the narrowing, shortening and offset. This example demonstrates that our method still exhibits excellent optimization capabilities for complex masks and has the potential to be applied to the production of complex masks.

\section{Conclusion} 
In this paper, the curvilinear mask boundary is represented by the periodic B-splines. The Delaunay triangulation and numerical integral are used to simulate the optical lithography imaging process. We establish an explicit relationship between the integral points and the control points of the spline curves. Based on the relationship, we further establish an explicit formula to calculate the gradient of objective functional for inverse lithography, and propose a steepest descent optimization algorithm based on the gradient. Various numerical experiments illustrate that our method exhibits excellent adaptability to the mask optimization problem, regardless of whether the mask has a single connected region or multiple separate regions. In the future, we will consider how to arrange initial control points more flexibly to give more freedom to the regions that need to be described more precisely. This will be helpful in improving the qualities of the inverse lithography problem.




\bibliographystyle{plain}
\bibstyle{abbrv}
    \bibliography{references}
\end{document}